\documentclass [11pt,reqno]{amsart}
\usepackage {amsmath,amssymb,epsfig,enumerate,verbatim}
\usepackage[all]{xy}

\newcommand{\C}{\mathbf{C}}

\newcommand{\Q}{\mathbf{Q}}
\newcommand{\R}{\mathbf{R}}

\newcommand{\Z}{\mathbf{Z}}
\newcommand{\N}{\mathbf{N}}
\renewcommand{\P}{\mathbf{P}}
\renewcommand{\H}{\mathbf{H}}

\newcommand{\fP}{\mathfrak{P}}

\newcommand{\fX}{\mathfrak{X}}

\newcommand{\fD}{\mathfrak{D}}

\newcommand{\hF}{\hat{F}}

\newcommand{\cA}{\mathcal{A}}

\newcommand{\cE}{\mathcal{E}}
\newcommand{\cF}{\mathcal{F}}

\newcommand{\cO}{\mathcal{O}}
\newcommand{\cV}{\mathcal{V}}
\newcommand{\cL}{\mathcal{L}}
\newcommand{\hcV}{\hat{\mathcal{V}}}
\newcommand{\hX}{\hat{X}}
\newcommand{\hcVqm}{\hat{\mathcal{V}}_{\mathrm{qm},X}}
\newcommand{\cVqm}{\mathcal{V}_{\mathrm{qm},X}}

\newcommand{\Pic}{\mathrm{Pic}\,}

\newcommand{\Ltwo}{\mathrm{L}^2}

\newcommand{\al}{\alpha}
\renewcommand{\a}{\alpha}

\newcommand{\mV}{\mathrm{V}}

\newcommand{\CNS}{\mathrm{C.NS}}
\newcommand{\NS}{\mathrm{NS}}
\newcommand{\Div}{\mathrm{Div}}
\newcommand{\e}{\varepsilon}

\newcommand{\om}{\omega}
\newcommand{\Om}{\Omega}

\newcommand{\fm}{\mathfrak{m}}

\newcommand{\la}{\lambda}

\renewcommand{\=}{:=}

\DeclareMathOperator{\ord}{ord}
\DeclareMathOperator{\pic}{Pic}

\numberwithin{equation}{section}       

\newtheorem{prop} {Proposition} [section]
\newtheorem{thm}[prop] {Theorem} 
\newtheorem{defi}[prop] {Definition}
\newtheorem{lem}[prop] {Lemma}
\newtheorem{cor}[prop]{Corollary}
\newtheorem{prop-def}[prop]{Proposition-Definition}

\newtheorem*{defi*}{Definition} 
\newtheorem*{thmA}{Theorem A} 
\newtheorem*{thmB}{Theorem B} 
\newtheorem*{corC}{Corollary C} 
\newtheorem*{thmD}{Theorem D} 
\newtheorem*{corE}{Corollary E} 
\theoremstyle{remark}

\newtheorem{rmk}[prop]{Remark}
\newtheorem*{rmk*}{Remark}
\newtheorem{qst}{Question}
 
\newtheorem{conj}{Conjecture}

\title{Holomorphic self-maps of singular rational surfaces}

\date{\today}

\author{Charles Favre}
\address{CNRS-Universit{\'e} Paris 7\\
  Institut de Math{\'e}matiques\\
  F-75251 Paris Cedex 05\\
  France\\
fax: +33 1 44 27 78 52
}
\email{favre@math.jussieu.fr}

\subjclass[2000]{Primary: 14J17, Secondary: 32H02}

\begin{document}

\begin{abstract}
We give a new proof of the classification of normal singular surface germs admitting non-invertible holomorphic self-maps and due to J.~Wahl.
We then draw an analogy between the birational classification of singular holomorphic foliations on surfaces, and the dynamics of holomorphic maps. Following this analogy, we introduce the notion of minimal holomorphic model for holomorphic maps. We give sufficient conditions  which ensure the uniqueness of such a model. 
\end{abstract}

\maketitle

\tableofcontents

%
%
%
%
\newpage
\section*{Introduction}

The study of dynamics of holomorphic maps over the complex projective space $\mathbb{P}^n(\C)$ has received a lot of attention since the beginning of the 90's (see the monographs~\cite{etats,japs} or the survey~\cite{sibony} and the references therein). It has been progressively realized that other varieties admit holomorphic maps with positive topological  entropy exhibiting complicated dynamics. For instance, many interesting examples  of invertible maps on surfaces have been recently described and studied~\cite{bedford-kim1,bedford-kim2,diller,cantat-k3,HV,K3mcm,projmcm,takenawa,zhang2}. Meanwhile, after the seminal work of Paranjape and Srinivas~\cite{paran}, an intense activity  grew up in classifying  projective  varieties admitting holomorphic maps~\cite{amerik,vandeven,beauville,cantat-homo,cantat-anosov,cantat,dabija,fuji,fuji1,fuji2,gurjar,naka-zhang,naka,zhang1}. In all the aforementioned works though, except for important contributions of Nakayama and D.-Q. Zhang~\cite{naka-zhang2,zhang3} and the recent thorough classification of endomorphisms on singular surfaces done by Nakayama~\cite{naka2}, varieties were assumed to be \emph{smooth}.
There exist however many examples of non-invertible maps that are holomorphic on a singular space, such as linear maps acting on quotients of Abelian surfaces, or holomorphic maps of weighted projective spaces, see~\cite{naka2} or~\S\ref{s:holo} in this paper for more examples. It is thus natural to try analyzing \emph{holomorphic maps acting on arbitrary singular compact varieties}.

The dynamics of dominant rational maps on $\mathbb{P}^n(\C)$ is also a subject growing rapidly.
A prominent feature of these maps is the existence of an indeterminacy set where they becomes singular, and which makes the analysis of the global dynamics extremely delicate. In particular,
a key step in extending results from the holomorphic to the rational case usually involves controlling the dynamics near this indeterminacy set. One possibility is to impose restrictions on the map near its indeterminacy set, see~\cite{DS1,DS2,dTV,GS}. On the other hand, if we insist on treating the case of an arbitrary rational map $F$, then one is naturally led to construct adapted models in which
$F$ has a relatively small indeterminacy set\footnote{the condition is usually stated in terms of the action of the map on the cohomology of the space and referred to as (algebraic) stability}. Constructing an adapted model is most of the time quite hard, even in concrete examples~\cite{bedford-kim3}. And we have at present very few general results proving the existence of an adapted model, see~\cite{diller-favre,polysmall}.
In any case, it has become clear that \emph{studying the action of a rational map in different models is
of great importance}, see~\cite{deggrowth,HPV,zhang3}.

\smallskip

In the present paper, we shall restrict ourselves to rational maps defined over $\mathbb{P}^2(\C)$ and give a description of the set of models in which a rational map becomes holomorphic. Our results hint at the existence of a birational classification of rational dynamical systems ``\`a la Mori''.
$$
\diamond
$$
In order to motivate our results, we shall first draw an analogy between the classification of holomorphic singular foliations on projective surfaces as developped in~\cite{brunella,mcquillan,mendes}, and the dynamics of rational maps\footnote{it is however not necessary to have any acquaintance with holomorphic foliations to understand our results}.

To make this analogy precise, we need to reformulate slightly some well-known results in the theory of foliations using
the terminology of~\cite{cantat-bir,deggrowth}. We refer to~\S\ref{s:WC} for precise definitions.
Following~\cite{manin}, one considers the space of all classes in the Neron-Severi space of all birational models of $\mathbb{P}^2$. Formally, one can do that in two ways, considering projective limits of classes under push-forward of morphisms, or injective limit of classes under pull-back. In the former case, one obtains the space $\NS(\fP)$ of so-called Weil classes; in the latter case, one gets the space $\CNS(\fP)$ of Cartier classes. Both spaces are infinite dimensional real vector spaces with $\CNS(\fP) \subset \NS(\fP)$. One can intersect a Weil class with a Cartier class and define the notion of pseudo-effective (we write $\al \ge0$ in this case) and nef class.
For later reference, we recall the existence of an intermediate space $\CNS(\fP) \subset \Ltwo(\fP) \subset \NS(\fP)$ on which the natural intersection form on $\CNS(\fP)$ extends, and which is complete for it.

The birational analysis of a foliated projective surface $(X,\cF)$ developed in~\cite{brunella,mcquillan,mendes} is based on the positivity properties of its cotangent line bundle that we denote by $(K_\cF)_X$. Looking at its induced class in all birational models of $X$, we get a Weil class $K_\cF \in \NS(\fP)$.
To define basic birational invariants of $\cF$ such as its Kodaira dimension, one needs to start with a good birational model $(X,\cF)$: formally speaking the foliation needs to have \emph{canonical} singularities (see~\cite{mcquillan})\footnote{by Seidenberg's theorem such a model always exists}. In terms of Weil classes, this means $(K_\cF)_X \le K_\cF$, or  in other words for any nef Cartier class $\al \in \NS(\fP)$,  one has $(K_\cF)_X \cdot \al \le K_\cF\cdot \al$.

It is natural to try to study the birational class of a rational dynamical system $F: \mathbb{P}^2 \dashrightarrow \mathbb{P}^2$ in terms of the positivity properties of the action of $F$ by \emph{pull-back}\footnote{the action of $F$ by pull-back allows one to recover the critical set of $F$ in all birational models} on $\CNS(\fP)$.
By analogy with the foliated case, we may say that $X$ is a good birational model for $F$ if its action  by pull-back on $\NS(X)$ is less positive than its action on $\CNS(\fP)$. In other words, $X$ is a good model if for any nef Cartier class $\al$ one has
$$
F^\sharp \al \= (F^* \al)_X \le F^* \al \text{ in } \CNS(\fP)~.
$$
It is a general fact that $F^\sharp \al \ge F^* \al$ for arbitrary nef Cartier 
classes~\cite[Proposition~1.13]{deggrowth}, so in a good birational model one has $F^\sharp = F^*$ on nef classes hence everywhere on $\NS(X)$.
This is equivalent to saying that $F$ induces a holomorphic map on $X$.
\begin{defi*}
Suppose $F: \mathbb{P}^2 \dashrightarrow \mathbb{P}^2$ is a dominant rational map. 
A (possibly singular) rational surface $X$ is a \emph{holomorphic model} for $F$ if the lift 
of $F$ to $X$ is holomorphic.
\end{defi*}
In the foliated case, once one has found a model with canonical singularities, one can construct many others by considering smooth models dominating it. One is thus lead to the question of
comparing these models, and constructing minimal or canonical models. 
It is a theorem of Brunella~\cite{brunella-minimal} that most foliations admit a unique smooth and minimal model with canonical singularities\footnote{exceptions include rational fibrations, Riccati foliations, and an isolated example}.
In another direction, a theorem of McQuillan~\cite{mcquillan} states that there always exists a model $X$ with cyclic quotient singularities in which $(K_\cF)_X$ becomes nef. One can then go further and either blow-down $X$ to a model with at most quotient or simple elliptic singularities in which $(K_\cF)_X$ becomes an ample line bundle; or prove that $X$ admits an invariant cycle of rational curves contracting to a
cusp singularity.\footnote{we refer to Section~\ref{s:special} for a definition of cusp and simple elliptic singularities}

In the dynamical case, the situation is quite different. Even a simple map like $(x,y) \mapsto (x,xy)$ does not admit any holomorphic model.
We are thus led to the following questions.
\begin{qst}\label{q1}
How do we detect when a rational map $F$ admits a holomorphic model?
\end{qst}
\begin{qst}\label{q2}
Suppose $F$ admits a holomorphic model. Is it possible to find one with only special singularities (like quotient singularities?)
\end{qst}
\begin{qst}\label{q3}
Describe all holomorphic models of a fixed rational map. 
In particular, does there always exist a minimal model? If yes, is it unique?
\end{qst}
Our aim is to give partial answers to these questions. We note that the mere existence of a holomorphic model is particularly restrictive for a rational map $F$. And as mentioned above, a (partial) classification of non-invertible maps admitting a holomorphic model has been carried out by Nakayama in~\cite{naka,naka2}.
One of the ingredients is the control of the singularity of a holomorphic model.
In the invertible case,  it is known though that there exists positive dimensional families of automorphisms on rational surfaces with positive entropy~\cite{bedford-kim1}.
$$\diamond$$
Before presenting our main results, we introduce some notations. If $F: \mathbb{P}^2 \dashrightarrow \mathbb{P}^2$ is a dominant rational map, we denote by $\deg(F)$ the degree of $F^{-1} L$ for a generic projective line $L$; and by $e$ the topological degree of $F$ that is the cardinality of $F^{-1} \{ p \}$ for a generic point $p\in \mathbb{P}^2$.
The sequence $\deg(F^n)$ is sub-multiplicative, thus $\la \= \lim_n \deg(F^n)^{1/n}$ exists. It is called the asymptotic degree of $F$. By Bezout's theorem, it satisfies $\la^2 \ge e$.

When $e < \la^2$, then there exist two unique nef classes $\theta_*, \theta^*$ in $\Ltwo(\fP)$ such that $F^* \theta^* = \la \theta^*$, and $F_* \theta_* = \la \theta_*$, see~\cite{deggrowth}. When $e=\la^2$, the situation is unclear, except in the birational case (see~\cite{diller-favre}) and in the polynomial case (see~\cite{polysmall}). Recall that a skew product is a rational map of $\mathbb{P}^1 \times \mathbb{P}^1$ preserving one of its natural rational fibrations.

Our answer to Question~\ref{q1} can be summarized in the following theorem.
\begin{thmA}\label{thm1}
Let $F: \mathbb{P}^2 \dashrightarrow \mathbb{P}^2$ be any dominant rational map. Suppose that $e< \la^2$, and that $F$ is not birationally conjugated to a skew product.
Then $F$ admits a holomorphic model iff both classes $\theta_*$ and $\theta^*$ are Cartier.
\end{thmA}
This result does not hold for skew product, see Section~\ref{s:skew}.

The theorem  was proved in the birational case in~\cite{diller-favre}, and under the assumption $e < \la$ in~\cite{DDG}. The case $e\ge \la$ is much harder, but our proof follows~\cite{DDG} with some arguments taken from~\cite{DJS,cantat}. It  is based on the analysis of the action of $F$ on $\Ltwo(\fP)$, and on the work of Sakai on rational surfaces of Halphen type~\cite{sakai}.
In the case $e=\la^2$, this action is much more complicated to describe. We thus content ourselves to state the following
\begin{conj}\label{c1}
Suppose $e = \la^2$. Then either $F$ is holomorphic in some model, or $F$ preserves a rational fibration, or $F$ is conjugated to a monomial map.
\end{conj}
The conjecture is true in the birational case by~\cite{diller-favre} and in the polynomial case by~\cite{polysmall}. In the latter case, the model in which $F^2$ becomes holomorphic is a weighted projective space. Note also that classification of holomorphic maps with $e= \la^2$ on rational surfaces is still open, see~\cite[\S 9]{naka2}.

\medskip

Next, we deal with Question~\ref{q2}.
The classification of normal surface germs $(X,0)$ which admit \emph{finite} holomorphic self-maps that are not invertible was obtained by J.~Wahl~\cite{wahl}. We give here an independent and more dynamical proof of his result.
The notion of klt (for Kawamata log-terminal) and lc (for log-canonical) singularities is recalled in Section~\ref{s:singular}. In dimension $2$, klt singularities are quotient singularities, and lc singularities correspond to minimal elliptic singularities, cusp singularities and their finite quotients, see~\cite{kawamata}.
\begin{thmB}[\cite{wahl}]\label{thm3}
Suppose $(X,0)$, $(Y,0)$ are two normal surface germs and $F: (X,0) \to (Y,0)$ is a finite holomorphic germ
which is not invertible. Denote by $JF$ the Weil divisor defined by the vanishing of
the Jacobian determinant of $F$.
\begin{enumerate}
\item
Suppose $(X,0)$ is klt. Then $(Y,0)$ is klt.
\item
Suppose $(X,0)$ is lc. Then $(Y,0)$ is lc.
\item
Suppose $X =Y$ and $JF$ is non empty. Then $(X,0)$ is klt.
\item
Suppose $X =Y$  and $JF$ is empty. Then $(X,0)$ is lc but not klt.
\end{enumerate}
\end{thmB}
Our method is based on looking at the action of $F$ on the space of all valuations on the local ring $\cO_X$, and on the thinness function\footnote{in the terminology of~\cite{ambro}, it is called the log-discrepancy, see Remark~\ref{rem: ambro} below} which loosely speaking measures the order of vanishing of a holomorphic two-form along exceptional prime divisors in any model. Our approach follows closely~\cite{valtree,eigenval}. Most of the material extends to higher dimension (see~\cite{hiro}).
In particular, the first three statements above hold  for finite maps between $\Q$-Gorenstein singularities of any dimension.
However the fourth statement relies on deeper properties of the valuation space that are not yet available in higher dimensions.
Of course, one would also like to remove the $\Q$-Gorenstein assumption in general. We refer to~\cite{defernex} for some ideas for dealing with this difficulty in a related context.
\begin{rmk*}
If one removes the assumption that $F$ is finite, the result clearly no longer holds. For any germ of singular surfaces $(X,0)$, $(Y,0)$ there exists a non finite holomorphic map $F: (X,0) \to (Y,0)$ with maximal generic rank. To see that, pick any resolution $\pi: \hat{Y} \to Y$, any point $p\in\pi^{-1} \{ 0 \}$
and a local biholomorphism $\phi$ from $(\C^2,0)$ to a local neighborhood of $p$ in $\hat{Y}$. Take any
map $h$ of maximal rank from $X$ to $(\C^2,0)$, and set $F = \pi \circ \phi\circ h$.
\end{rmk*}
As a consequence of Theorem~B above, we shall prove
\begin{corC}\label{thm2}
 Suppose $F: \mathbb{P}^2 \dashrightarrow \mathbb{P}^2$ is a dominant rational map
which is holomorphic in some (possibly singular) model of $\mathbb{P}^2$. Then there exists a holomorphic model $X$ for $F$ with quotient or cusp singularities.
\end{corC}
In fact, as a corollary of his classification, Nakayama proved that one can rule out the case where the singularities are cusp, see~\cite[Theorem~1.4]{naka2}.

Note that in the birational case, because any normal singularity admits a minimal desingularization (see~\cite{laufer}), then one can actually find a \emph{smooth} holomorphic model.
In Section~\ref{s:holo}, we present a number of examples of non-invertible holomorphic maps preserving klt or lc singularities.

\medskip

Finally, we answer Question~\ref{q3} in a special case.
Recall that a model is a rational (singular) surface together with a bimeromorphic map
 $X \dashrightarrow \mathbb{P}^2$ defined up to isomorphism.
\begin{thmD}[Uniqueness of minimal model]
Suppose $F: \mathbb{P}^2 \dashrightarrow \mathbb{P}^2$ is a rational dominant map which is holomorphic in some model and satisfies $e< \la^2$. Then there exists a holomorphic model $X_0$ for $F$ such that for any other holomorphic model $X$, the natural map $X \to X_0$ is regular.
\end{thmD}
When $F$ is not invertible, one can rely on the classification of non-invertible endomorphisms of rational surfaces with $e<\la^2$ given in \cite[\S 9]{naka2}, and deduce that $X_0$ has only quotient singularities. Theorem~B only implies that the singularities of $X_0$ are lc.

On the other hand, when $F$ is invertible, it automatically admits many holomorphic models, as one may blow-up any periodic cycle. Theorem~D is thus particularly interesting in this case. 
By~\cite{DJS} the singularities of $X_0$ are then either rational or elliptic Gorenstein. 
Let us state the following amusing corollary of the previous theorem.
\begin{corE}
Under the same assumption of the previous theorem, either $F$ preserves a fibration, or it admits no invariant curve on its minimal model.
\end{corE}
Indeed if $F_*(C) = C$ for some curve then we would have $\la ( C \cdot \theta^*) = C \cdot \theta^*$, and one could either contract $C$ or find an invariant fibration by applying Theorem~\ref{t:contract} below.
This corollary can be seen as a discrete dynamical version of Jouanolou's theorem which states that
a holomorphic foliation which has infinitely many invariant compact curves must be a fibration.
We refer to the recent preprint~\cite{cantat-fib} for more general results in this direction.

\medskip

We would like to conclude this introduction by stating the following
\begin{qst}
Does there exist an algorithm for checking if $F$ admits a holomorphic model?
Same question for knowing when $F$ preserves a fibration.
\end{qst}
In the birational case, the question reduces to build an algorithm for constructing a model 
in which no contracted curve by $F$ is eventually mapped to a singular point (see~\cite{diller-favre}). 

$$\diamond$$
The paper is divided into three sections. The first one is devoted to the necessary material on valuation spaces of normal surface singularities. The presentation is quite sketchy as most 
of the theory is a direct generalization of the smooth case that we already developed in great detail in~\cite{valtree} with M.~Jonsson. In the second section, we prove Theorem~B and Corollary~C and describe examples of holomorphic self-maps both in a local context and on rational singular surfaces.
The third and last part is of a more global nature. We prove there a refined version of Theorem~A on the characterization of holomorphic maps, see Theorem~\ref{t:contract}.  We then prove Theorem~D.

$$\diamond$$

\noindent {\bf Acknowledgements:}
it is a pleasure to thank my collaborator and friend Mattias Jonsson to whom I owe so much. The material presented here is a natural extension of our joint work on valuation spaces.
I also thank Professors Nakayama and Wahl for letting me know about their papers~\cite{naka2} and~\cite{wahl} respectively, and Sebastien Boucksom and Jeffrey Diller for many valuable comments.

\section{Normal singular surfaces}\label{s:singular}
Our aim is to extend the results of~\cite{valtree} to normal surface singularities. To deal with general non $\Q$-Gorenstein singularities, we use Mumford's numerical pull-back of Weil divisors on surfaces.

\subsection{The valuation space of a normal surface singularity}\label{s:valuation}
~
\medskip

\noindent \emph{Basics}.
Let $(X,0)$ be a germ of \emph{normal} surface singularity. We denote by $\cO_X$ the ring of holomorphic germs defined on $(X,0)$. It is a local ring with maximal ideal $\fm_X\= \{ f\in\cO_X,\, f(0) =0 \}$. Following~\cite{valtree}, we introduce the set $\hcV$ of all valuations $\nu: \cO_X \to \R_+\cup\{+\infty\}$ which are trivial on $\C$, non identically $+\infty$ and positive on $\fm_X$. It is a cone over the subset $\cV_X$ of valuations normalized by the condition $\nu(\fm_X) \= \min \{ \nu(f), \, f \in \fm_X \} = +1$. With the topology of the pointwise convergence, $\cV_X$ is compact and $\hcV_X$ is locally compact.

A morphism $\mu: \hX \to X$ is called a resolution if $\hX$ is smooth and $\mu$ is a regular bimeromorphism.  A resolution $\mu'$ dominates $\mu$ if the map 
$\mu^{-1}\circ \mu'$ is regular.
Any germ of normal surface singularity admits a unique minimal resolution $\mu_0 : \hX_0 \to X$ characterized by the fact that any other resolution $\mu$ dominates $\mu_0$, see~\cite{laufer}.
A resolution for which $\mu^{-1} (0)$ has only normal crossing singularities,  all its irreducible components are smooth, and the intersection of any two components is a single point, is called a \emph{good resolution}. Any resolution is dominated by a good resolution. As we shall see below good resolutions are adapted to
interpret dual graphs as subsets of $\cV_X$.

A valuation $\nu$ is \emph{divisorial} if one can find a resolution $\mu$, a prime divisor $E \subset \mu^{-1} (0)$, and a positive real number $c>0$ such that $\nu(f) = c \times \ord_E(f \circ \mu)$ for any $f \in \cO_X$. In this case we shall write $\nu = c\times \mu_* \ord_E$. When we want to emphasize that $\nu$ is associated to the prime divisor $E$, we write $\nu = \nu_E$. With this notation, we always assume $\nu_E$ to be normalized by $\nu_E(\fm_X) = +1$.
To any divisorial valuation $\nu = \nu_E$ is attached a positive integer $b(\nu) = b_E$ called the \emph{generic multiplicity}. It is defined by $b_E \= \ord_E (\mu^* \fm_X)$, so that the valuation $b_E^{-1} \mu_* \ord_E$ is normalized.

\medskip

\noindent \emph{Metric and affine structure}.
Fix a good resolution $\mu: \hX \to X$, and  denote by $\cE(\mu)$ the finite set of prime divisors $E \subset \mu^{-1}(0)$. Define the dual graph $\Gamma(\mu)$ to be the following finite simplicial graph. Vertices are in bijection with $\cE(\mu)$ and  edges with the set of singular points of $\mu^{-1}(0)$.
Two vertices $E$ and $E'$ are connected by an edge $p$ if $E\cap E' =p$.
For convenience, we fix a parameterization by the segment $[0,1]$ of any edge.

The graph $\Gamma(\mu)$ naturally embeds in the normalized valuations space $\cV_X$ as follows. We map any vertex $E \in \cE(\mu)$ to the valuation $\nu_E \= b_E^{-1} \mu_* \ord_E$.  Pick an edge $p = E \cap E'$, and take local coordinates $(z,z')$ at $p$ such that $E= \{ z=0 \}$, $E' = \{ z' =0 \}$. 
For any $ t \in [0,1]$, and any holomorphic germ $g(z,z') = \sum g_{ij} z^i (z')^j$ at $p$, 
we define:
\begin{equation}\label{e:monom}
\nu_t (g) = \min\left\{ \frac{i}{b_E}t + \frac{j}{b_E'} (1-t),\, g_{ij} \neq 0 \right\}~. 
\end{equation}
Such a valuation is called monomial as it is determined by its values on the monomials $z$ and $z'$.
Note that $\nu_t$ is divisorial iff $t$ is a rational number.
Now we map the point in the edge $p$ given by the parameter $t$ to the valuation $\mu_* \nu_t$ defined by $\mu_* \nu_t (f) = \nu_t (f \circ \mu)$ for all $f\in \cO_X$. By construction $\mu_* \nu_t$ is normalized and belongs to $\cV_X$, and the map $\Gamma(\mu) \to \cV_X$ is continuous. It is not hard to prove that this map is actually injective. In the sequel, we shall always identify the simplicial graph $\Gamma(\mu)$ with its image in $\cV_X$. 

Any valuation which is proportional to a valuation belonging to the dual graph of some good resolution is called \emph{quasi-monomial}. The set of all quasi-monomial valuations in $\hcV_X$ (resp. $\cV_X$) will be 
denoted by $\hcVqm$ (resp. by $\cVqm$).

A \emph{segment} in $\cV_X$ is a closed subset which coincides with the edge of some dual graph $\Gamma(\mu)$ of some good resolution. In an equivalent way, a segment consists of all valuations  
of the form $\mu_* \nu_t$, $t\in [0,1]$ with $\nu_t$ as in~\eqref{e:monom}.

Fix a good resolution $\mu$. On $\Gamma(\mu)$, we put the unique metric $d_\mu$ such that for any edge $E \cap E'$ one has
$d_\mu (\mu_* \nu_t, \mu_* \nu_{t'}) = |t-t'|/(b_E b_{E'})$ in the parameterization given 
by~\eqref{e:monom}. The $d_\mu$-length of the edge associated to the point $E \cap E'$ is thus equal to $(b_E b_{E'})^{-1}$. If $\mu'$ dominates $\mu$, then the regular map $\mu^{-1} \circ \mu$ is a composition of point blow-ups. Arguing inductively, one shows that $\Gamma(\mu)$ embeds naturally into $\Gamma(\mu')$ and that the embedding is isometric for the metrics $d_\mu$, $d_{\mu'}$.
The metrics $d_\mu$ thus patch together. In other words, there exists a unique metric $d$ on $\cVqm$ such that its restriction to any $\Gamma(\mu)$ equals $d_\mu$ for all good resolutions.
 
Note that the topology induced by the metric $d$ on $\cVqm$ is not locally compact, and is stronger than the topology of the pointwise convergence. 
Note also that the metric $d$ induces on any segment of $\cVqm$ an affine structure. A map $f: \cVqm \to \R$ is \emph{affine} on $\Gamma(\mu)$ if its restriction to any edge is affine that is $f (t) = at +b$ in the parameterization given by~\eqref{e:monom} for some $a,b \in \R$.
%
\medskip

\noindent \emph{Global topology of $\cV_X$}.
Pick a valuation $\nu \in \hcV_X$, and a resolution $\mu: \hX \to X$. 
For any (scheme-theoretic) point $p$ in $\mu^{-1}(0)$, then $\nu$ induces a valuation on the local ring $\cO_{\hX,p}$ and one can show that there exists a unique point for which $\nu$ is positive on the maximal ideal $\fm_{\hX,p}$. This point is called the \emph{center} of $\nu$ in $\hX$. When the center is positive dimensional, then it is a prime divisor $E$ and $\nu$ is a divisorial valuation proportional to $\mu_* \ord_E$.

Suppose now that $\mu$ is a good resolution. One can construct a natural retraction $r_\mu: \cV_X \to \Gamma(\mu)$ in the  following way. On $\Gamma(\mu)$ we set $r_\mu = \mathrm{id}$.
Pick $\nu\in \cV_X\setminus\Gamma(\mu)$. Then the center of $\nu$ is  a point. If
it is not the intersection of two exceptional primes, then it is included in a unique prime $E$. In this case, we set $r_\mu(\nu) \= \nu_E = b_E^{-1} \mu_*\ord_E$. Otherwise, the center of $\nu$ is equal to $p = E\cap E'$. We fix local coordinates $z,z'$ at $p$ with $E = \{ z=0 \}$, $E'= \{ z'=0 \}$, and set $r_\mu(\nu)$ to the unique monomial valuation in the coordinates $z,z'$ which agrees with $\nu$ on $z$ and $z'$. In the notation of~\eqref{e:monom}, we have $r_\mu(\nu)\= \mu_* \nu_t$ with $\nu_t(z) = \nu(z)$, $\nu_t(z') = \nu(z')$. The map $r_\mu$ is continuous for the topology of the pointwise convergence.

Suppose $\mu'$ is another good resolution which dominates $\mu$. Then $\mu'$ is the composition of $\mu$ with a finite sequence of point blow-ups. In the case this sequence is reduced to the blow-up at one point $p$, then either $p= E \cap E'$ for some $E, E' \in \cE(\mu)$ and $\Gamma(\mu')=\Gamma(\mu)$; or $p$ belongs to a single $E\in\cE(\mu)$ and  $\Gamma(\mu')$ is obtained from $\Gamma(\mu)$ by adding one segment attached to the vertex corresponding to $E$. In any case, we see that there is a natural retraction map
$\Gamma(\mu') \to \Gamma(\mu)$. This retraction coincides with $r_\mu$.

Note also that these retraction maps are compatible in the sense that $r_\mu = r_\mu|_{\Gamma(\mu')} \circ r_{\mu'}$. We can thus consider the following topological space $\varprojlim_\mu \Gamma(\mu)$.
It is a compact space and there is a natural continuous map $\cV_X \to \varprojlim_\mu \Gamma(\mu)$.
It follows from a theorem of Zariski  that this map is bijective (see~\cite{vaquie}). Hence
$\cV_X$ is homeomorphic to $\varprojlim_\mu \Gamma(\mu)$.

This presentation of $\cV_X$ has several consequences. Firstly, $\cVqm$  and the set of divisorial valuations are both dense in $\cV_X$. Secondly, any finite graph $\Gamma(\mu)$ is a retract
of $\cV_X$. In fact, one can show there is actually a strong deformation retract from $\cV_X$ to any $\Gamma(\mu)$, see~\cite{berkovich,thuillier}. In particular, $\cV_X$ has the homotopy type of the dual graph of the minimal resolution of $(X,0)$.

One can also build the injective limits of all graphs $\Gamma(\mu)$. In this way, we get back the space
$\cVqm$.  We may transport the affine structure of the graphs to $\cVqm$ using this presentation, and make the following definition.
A  \emph{piecewise affine function} $f$ on $\cV_X$ is a real-valued function such that there exist
a good resolution $\mu$ and a continuous function $g$ on $\Gamma(\mu)$ which is affine on each the edge of $\Gamma(\mu)$ for the metric $d$, and such that $f(\nu) = g \circ r_\mu(\nu)$.

\medskip

\noindent \emph{Transform of divisors}.
Fix any good resolution $\mu : \hX \to X$. The intersection form on $\Div(\hX) \= \sum_{E \in \cE(\mu)} \R [E]$ is negative definite, hence for any prime divisor $E\in \cE(\mu)$ there exists a (unique) divisor $Z_E\in \Div(\hX)$ such that
\begin{equation}\label{e:dual}
 Z_E \cdot Z = \ord_E(Z)
\end{equation}
for all $Z \in \Div(\hX)$. The following lemma is classical, see~\cite[Corollary~4.2]{kollar-mori}.
\begin{lem}\label{l:neg}
For any good resolution $\mu$ and any prime divisor $E\in \cE(\mu)$, one has
$\ord_{E'}(Z_E) <0$ for all $E'\in\cE(\mu)$.
\end{lem}
Suppose one is given a Cartier divisor $C$ in $X$ defined by the equation $\{ f=0 \}$ with $f \in \cO_X$.
Then for any $\nu \in \hcV_X$ one can set $\nu(C) \= \nu(f)$. This does not depend on the choice of the equation $f$. If $C$ is merely a Weil divisor, one can still define $\nu(C)$ using Mumford's pull-back as follows.
Pick a resolution $\mu: \hX \to X$, and let $p$ be the center of $\nu$ in $\hX$.
Then $\nu = \mu_* \nu'$ for some valuation $\nu'$ defined on the regular local ring $\cO_{\hX,p}$.
Denote by $\mu_*^{-1}(C)$ the strict transform of $C$ by $\mu$.
There exists a unique divisor $Z_C\in\Div(\hX)$ such that $\mu_*^{-1}C + Z_C$ is numerically trivial, i.e. $(\mu_*^{-1}C+Z_C) \cdot E =0 $ for all $E\in \cE(\mu)$. We pick an equation $g \in \cO_{\hX,p}$ of $\mu_*^{-1}C + Z_C$ near $p$, and we set $\nu(C) \= \nu' (g)$.
This definition is compatible with the standard one in the Cartier case. Indeed if $C$ is the divisor of $f\in \cO_X$, then $Z_C = \sum_{E \in \cE(\mu)} \ord_E(\mu^*f) [E]$.
\begin{prop}\label{p:maps}
For any effective Weil divisor $C$, the function $g_C(\nu)= \nu(C)$ on $\hcV$ is $1$-homogeneous, continuous and non-negative with values in $\R_+\cup \{ + \infty \}$. For any good resolution $\mu$, one has $g_C \circ r_\mu \le g_C$ and the restriction $g_C|_{\Gamma(\mu)}$  is piecewise affine.

If $C'$ is any other Weil divisor having no common component with $C$, then
the function $\min \{ g_C, g_{C'} \}$ is piecewise affine on $\cV_X$.
\end{prop}
\begin{proof}
For the first statement, one only needs to prove $g_C \circ r_\mu \le g_C$ on any $\Gamma(\mu')$ for $\mu'$ dominating $\mu$. The proof is an easy induction on the number of exceptional components of $\mu^{-1} \circ \mu'$.
To prove that the restriction of $g_C$ is piecewise affine on any segment, we
may suppose  that the total transform of $C$ has only normal crossing singularities. Each edge of $\Gamma(\mu)$ corresponds to a point $p$, where the divisor $\mu_*^{-1}C+Z_C$
is defined by a monomial $z^a(z')^b = 0$ for some local coordinates $z,z'$. Then by~\eqref{e:monom},
$g_C(\nu_t) = at/b_E + b(1-t)/b_{E'}$ which is clearly affine in $t$.
The last statement is proved in the same way, by picking a resolution of the divisor $C+C'$.
\end{proof}
\smallskip
\subsection{Divisors on $\fX$}\label{s:potential}
~

\noindent \emph{Cartier and Weil classes}.
The set of all good resolutions forms an inductive set: for any two such resolutions $\mu_1, \mu_2$ there exists a third one $\mu_3$ dominating the two. One can thus consider the injective and projective limits of $\Div(X_\mu)$ over all resolutions $\mu: X_\mu\to X$.
\begin{defi}
A Weil divisor $Z$ is a collection of divisors $Z_\mu \in \Div(X_\mu)$ such that $\varpi_* Z_{\mu'} = Z_\mu$ whenever $\varpi \= \mu^{-1} \circ \mu'$ is regular.

A Cartier divisor is a Weil divisor such that there exists a good resolution $\mu_0$ for which 
$Z_\mu  = \varpi^* Z_{\mu_0}$ whenever $\varpi \= \mu_0^{-1} \circ \mu$ is regular.
\end{defi} 
The set of all Weil (resp. Cartier) divisors will be denoted by $W(\fX)$ (resp. $C(\fX)$).
For a Cartier divisor, a good resolution $\mu_0$ as above is called a \emph{determination} of $Z$. One also says that $Z$ is determined in $X_{\mu_0}$.

The set of all Weil classes can be endowed with the projective limit topology. In this topology, a sequence $Z_n$ converges to $Z$ iff $Z_{n,\mu} \to Z_\mu$ in the finite dimensional vector space $\Div(X_\mu)$ for all good resolutions $\mu$. 

There is a canonical bijection between functions on the set of divisorial valuations in $\cV_X$ and Weil divisors. To any such function $g$ and any good resolution we attach the divisor $Z_\mu(g) \= \sum_{E\in \cE(\mu)} g(\nu_E) b(\nu_E) [E]$. We write $Z(g)$ for the Weil divisor associated to $g$, and $g_Z$ for the function associated to $Z$. 
A Weil divisor is Cartier iff the function $g_Z$ is the restriction to divisorial valuations of a piecewise affine function as defined in Section~\ref{s:valuation}. The topology on Weil classes corresponds to the pointwise convergence on functions.

There is a natural pairing on $C(\fX)\times W(\fX)$. Pick $Z \in C(\fX)$, $W \in W(\fX)$ and a good resolution $\mu$ such that $Z$ is determined in $X_\mu$. Then one can set $Z \cdot W \=  Z_{\mu} \cdot W_{\mu}$. This does not depend on the choice of the determination.
For each $\mu$ the intersection product is negative definite on $\Div(X_\mu)$ hence the induced intersection form on $C(\fX)$ is also negative definite.

\smallskip

Fix  a resolution $\mu: X_\mu \to X$. A divisor $Z\in \Div(X_\mu)$ is effective if
$\ord_E(Z) \ge 0$ for any prime divisor $E\in\cE(\mu)$. In this case we write $Z\ge 0$. A divisor is \emph{nef}\footnote{formally speaking $Z$ is relatively nef} if $Z \cdot E \ge 0$ for any $E\in\cE(\mu)$.
 The cone of nef classes is generated by  the divisors $\{ Z_E\} _{E \in \cE(\mu)}$ as defined 
in~\eqref{e:dual}. If $Z\neq 0$ is nef, then $-Z$ is effective and $\ord_E(Z) <0$ for all $E$ by Lemma~\ref{l:neg}. A Weil class $Z$ is effective (resp. nef) if all its incarnations $Z_\mu \in \Div(X_\mu)$ are effective (resp. nef).

\medskip
\noindent \emph{Dual divisors}.
Pick an arbitrary valuation $\nu\in \cV_X$, and a Cartier divisor $W$ determined in some model $X_\mu$.
Let $p$ be the center of $\nu$ in $X_\mu$. 
Then  $\nu$ induces a valuation $\nu'$ on the local ring at $p$, and
we set $\nu(W) \= \nu'(f)$ where $f$ is a local equation for $W$ at $p$.
\begin{prop}\label{p:depend}
For any $\nu \in \cVqm$ there exists a unique 
nef Weil divisor $Z_\nu$ such that
$Z_\nu \cdot W = \nu (W)$ for any Cartier divisor $W$.
%
Moreover, for any fixed $\nu_0 \in \Gamma(\mu)$
there exist positive constants $c_0, c_1 >0$ such that:
\begin{equation}\label{e:bdd}
c_0 \, Z_{\nu_0} \le Z_{\nu} \le c_1 \, Z_{\nu_0}~,
\end{equation}
for all $\nu \in \Gamma(\mu)$.
\end{prop}
\begin{proof}
The construction of $Z_\nu$ in the case $\nu$ is divisorial was done above, see~\eqref{e:dual}.
Let us indicate how to extend it to arbitrary quasi-monomial valuations.
Pick a good resolution $\mu: \hX \to X$,  and a point $p$ at the intersection of two exceptional primes $E_0$ and $E_1$ with associated divisorial valuations $\nu_0$ and $\nu_1$. For any $t\in[0,1]$, denote by $\nu_t$ the unique valuation in the segment $[\nu_0,\nu_1]$ at distance $t\times d(\nu_0,\nu_1)$ from $\nu_0$ for the metric $d$ defined on $\cVqm$ in Section~\ref{s:valuation}.
Let $g_0$ and $g_1$ be the functions on $\cV_X$ associated to the two Cartier divisors determined by $Z_{\nu_0}$ and $Z_{\nu_1}$ in $\hX$. For each $t$, set
\begin{equation}\label{e:def-lapl}
g_t = t g_1 + (1-t) g_0 - h_t
\end{equation}
on $\cVqm$ where $h_t$ is the unique continuous function on $\cVqm$ satisfying: $h_t \circ r_\mu = h_t$; $h_t \equiv 0$ on $\Gamma(\mu) \setminus ]\nu_0,\nu_1[$; $h_t$ is affine on the two segments $[\nu_0, \nu_t]$ and $[\nu_t, \nu_1]$; and
$h_t(\nu_t)  = (1-t) d(\nu_t,\nu_0)$.
We claim that $Z_{\nu_t}\= Z(g_t)$.

The proof of the claim can be done as follows. By continuity, it is sufficient to prove it for rational $t$, that is for a divisorial valuation $\nu_t$. By induction, one is  reduced to  the case one blows up once the point $E_0\cap E_1$. Let $E$ be the exceptional divisor. Then $b_E = b_0 + b_1$ hence $t=b_1/(b_0+b_1)$, and $Z_E = Z_{E_0} + Z_{E_1} - E$. The function $g_E$ is associated to the divisor $b_E^{-1} Z_E$, hence
$g_E = b_0 b_E^{-1} g_0 + b_1 b_E^{-1} g_ 1- b_E^{-1} h = tg_1+ (1-t) g_0 - b_E^{-1} h$ where $h$ is the function associated to the Cartier class determined by $E$.
To conclude that $b_E^{-1} h$ corresponds to $ h_t$, one only needs to show that the values of the two functions at $\nu_E$ are equal.
But $b_E^{-1} h_E(\nu_E) = b_E^{-2} = t (1-t) d(\nu_0, \nu_1)$, and $h_t(\nu_E) = h_t(\nu_0) + (1-t) d(\nu_t,\nu_0) = 0 + t (1-t) d(\nu_0, \nu_1)$.

%
Finally, to prove~\eqref{e:bdd} it is sufficient to find $c_0, c_1 >0$ such that
\begin{equation}\label{e:pf-fndtm}
-c_0 \le g_t \le -c_1
\end{equation}
for all $0 \le t \le 1$.
By Lemma~\ref{l:neg}, $g_\varepsilon$ is negative, and $g_\varepsilon \circ r_\mu = g_\varepsilon$, for $\varepsilon =0,1$. Thus~\eqref{e:pf-fndtm} holds for $t \in \{ 0,1\}$ with $ c_0 = - \min_{\Gamma(\mu)} \{ g_0, g_1 \}$ and $ c_1 = - \max_{\Gamma(\mu)} \{ g_0, g_1 \}$. Now for any $t$, one has $0 \le h_t \le c_2:= d(\nu_0,\nu_1)$. Since $g_t \circ r_\mu = g_t$, we conclude that $ -c_0-c_2 \le g_t \le - c_1$.
\end{proof}

\smallskip
\subsection{The Jacobian formula}
~
\medskip

\noindent \emph{Holomorphic germs}.
Let $(X,0)$, $(Y,0)$ be two germs of normal surface singularity and pick $F: (X,0) \to (Y,0)$ a finite holomorphic germ. For any $\nu \in \hcV_X$, we define $F_*\nu \in \hcV_Y$ by the formula $F_*\nu (f) = \nu ( f \circ F)$ for any $f \in \cO_Y$. The map $F_*: \hcV_X \to \hcV_Y$ is continuous. We further define the \emph{local contraction rate} at $\nu \in \hcV_X$ to be  $c(F,\nu) \= F_*\nu (\fm_Y) = \min \{ \nu ( f \circ F ) , \, f\in \fm_Y \} \in ]0, + \infty]$. As $F$ is finite, $c(F,\nu) < +\infty $ for all valuations, so that $c(F,\cdot)$ is continuous. On the other hand, note that $F^* \fm_Y \subset \fm_X$ implies $ c(F, \cdot) \ge 1$. We thus have a continuous map on normalized valuation spaces $F_\bullet: \cV_X \to \cV_Y$, defined by 
$F_\bullet \nu = c(F,\nu)^{-1}\times F_*\nu$.
\begin{lem}\label{l:finite}
If the map $F: (X,0) \to (Y,0)$ is finite, then  $F_*: \hcV_X \to \hcV_Y$ and
$F_\bullet : \cV_X \to \cV_Y$ are surjective. Moreover, any fiber $F_\bullet^{-1} \{ \nu \}$
is finite of cardinality bounded above by the topological degree of $F$.
\end{lem}
The proof is the same as in~\cite[Proposition~2.4]{eigenval}.
%

\noindent \emph{Thinness}.
We now introduce on $\hcV_X$ an important function called the thinness that measures the order of vanishing of the canonical divisor along a valuation. As a canonical divisor of $(X,0)$ might not be $\Q$-Cartier we follow~\cite[Chapter~4]{kollar-mori}.

A holomorphic $2$-form $\om$ on $(X,0)$ is by definition a holomorphic $2$-form on $X\setminus \{ 0\}$ equal to the restriction of some holomorphic $2$-form $\Om$ on $(\C^n,0)$ given some embedding of $X$ into $(\C^n,0)$. Its zero-locus $\Div(\om)$ defines a Weil divisor in $(X,0)$.
Suppose $F: (X,0) \to (Y,0)$ is a  holomorphic map between two normal surface singularities. Then we can embed $X$ and $Y$ into $\C^n$ and $\C^m$ and find a holomorphic map $\tilde{F}: \C^n \to \C^m$ such that
$\tilde{F}|_X = F$. If $\om$ is a holomorphic $2$-form on $Y$, then we set $F^* \om \= \tilde{F}^* \Om$.
The divisor of $F^*\om$ is related to the divisor of $\om$ by the formula
$$\Div(F^*\om) = F^*\Div(\om)+ JF~,$$ where $JF$ is the ramification locus, i.e. the Weil divisor defined by the vanishing of the Jacobian determinant of $F$ on the set of smooth points of $X$.

The \emph{relative canonical divisor} $K\mu \in \Div(\hX)$ is then defined by the relation
$$
\Div(\mu^* \om) = K\mu + \mu_*^{-1} \Div(\om) + Z_\om~,
$$
where $Z_\om$ is the unique divisor supported on $\mu^{-1}(0)$ such that
$\mu_*^{-1} \Div(\om) + Z_\om$ is numerically trivial. In other words, $\mu_*^{-1} \Div(\om) + Z_\om$
is Mumford's pull-back of $\Div(\om)$ by $\mu$.
Note that if $E$ is a prime divisor of genus $g$, the adjunction formula on $\hX$ yields the relation
$g = 1+ \frac12 (E^2 + E \cdot K\mu)$. This relation shows that the relative canonical bundle does not depend on the choice of holomorphic $2$-form.
\begin{prop}
There exists a unique function $A: \hcV \to \R \cup \{ + \infty\}$ such that
\begin{enumerate}
 \item $A(t\nu) = t \times A(\nu)$ for all $t>0$;
 \item for any good resolution $\mu:\hX \to X$ and any prime divisor $E \in \cE(\mu)$, then
$A(\mu_*\ord_E) = 1+ \ord_E(K\mu)$;
\item for any good resolution $\mu:\hX \to X$, the function $A$ is continuous on $\Gamma(\mu)$ and affine on each of its edges;
\item $A$ is a lower semi-continuous function, and $A = \sup_\mu A \circ r_\mu$ where the supremum is taken over all good resolutions.
\end{enumerate}
\end{prop}
The value $A(\nu)$ is called the thinness of the valuation $\nu$.
\begin{rmk}\label{rem: ambro}
In Mori's theory, $\ord_E(K\mu)$ is called the discrepancy of $X$ along $E$, and the $1+\ord_E(K\mu)$ is referred to as the log-discrepancy, see~\cite{ambro}.
The introduction of the thinness/log-discrepancy is motivated by the fact that the discrepancy does not vary nicely on the space of \emph{normalized} valuations.
\end{rmk}

\begin{proof}
The uniqueness of $A$ is clear.

To construct $A$ we proceed as follows.
Pick a good resolution $\mu$. If $\nu$ is a vertex of $\Gamma(\mu)$ then
$\nu = b_E^{-1} \mu_* \ord_E$ for some $E \in \cE(\mu)$ and we set
$A_\mu(\nu) = b_E^{-1}(1+\ord_E(K\mu))$. We then extend $A_\mu$ to
$\Gamma(\mu)$ by imposing that it is affine on each edge.

We claim that the functions $A_\mu$ patch together and define a global function $A$ on $\cVqm$. To see that, one only needs to check that $A_\mu$ and $A_{\mu'}$ coincide
when $\mu'$ is obtained from $\mu$ by blowing up a single point. If $\varpi \= \mu^{-1} \circ \mu'$ and $E$ is the exceptional divisor of $\varpi$, then
$K\mu' = \varpi^* K\mu + E$. A direct computation leads to $A_{\mu'} = A_{\mu}$ on all vertices
of $\Gamma(\mu')$ belonging to $\Gamma(\mu)$ which implies $A_{\mu'} = A_{\mu}$ on $\Gamma(\mu)$. We then extends $A$ to $\hcVqm$ by imposing~(1).

The function one has obtained clearly satisfies (1), (2) and (3).
For this function $A$, property (4) is a consequence of the inequality $A(\nu) \ge A \circ r_\mu(\nu)$ for any good resolution and any $\nu\in \hcVqm$. As before, one reduces the proof of this inequality to the case $\nu$ is the divisorial valuation associated to the blow-up at a point $p\in\mu^{-1}(0)$. When $p$ is the intersection of two exceptional primes, then $\nu$ belongs to $\Gamma(\mu)$ and thus $r_\mu(\nu) = \nu$. Otherwise, it belongs to a unique prime $E$, and one has $A(\nu) = A(\nu_E) + b_E^{-1} > A(\nu_E) = A(r_\mu(\nu))$.
\end{proof}
From properties (3) and (4), one deduces the following important fact:
\begin{cor}\label{cor:thin}
The thinness is bounded from below on $\cV_X$, and attains its minimum.
The subset where the minimum is attained is a finite union of segments in $\cVqm$
which is included in the intersection $\cap_\mu \Gamma(\mu)$ taken over all good resolutions of $(X,0)$.
\end{cor}

\medskip

\noindent \emph{The Jacobian formula}.
Recall that $JF$ is the Weil divisor defined on the regular part of $X$ by the vanishing of the Jacobian determinant of $F$.
\begin{prop}
Let $F: (X,0) \to (Y,0)$ be a finite holomorphic germ between two normal surface singularities.
Then one has:
\begin{equation}\label{e:jac}
 A(F_*\nu) = A(\nu) + \nu(JF)~,
\end{equation}
for any valuation $\nu$ centered at $0$ in $X$. 
\end{prop}
\begin{proof}
It is sufficient to prove the formula for a divisorial valuation. 
We thus suppose that $\mu$ (resp. $\mu'$) is a good resolution of $X$ (resp. $Y$);
that $\hF \= (\mu')^{-1} \circ F \circ \mu$ is regular; that $\nu = \mu_* \ord_E$ for some
$E\in \cE(\mu)$; and that  $\hF(E) = E'$ for some $E'\in \cE(\mu')$. We denote by $k\in\N^*$ the unique integer such that $\hF^*[E']-k[E]$ has no support on $E$.

Pick $\om$ a holomorphic $2$-form on $Y$, and let $C$ be the (Weil) divisor of $\om$. 
Set $(\mu')^* C = (\mu'_*)^{-1}(C) + Z$ where $(\mu'_*)^{-1}(C)$ is the strict transform of $C$ and $Z$ is the unique exceptional divisor such that $(\mu')^* C$ is numerically trivial. Then the holomorphic $2$-form $\Om \= \hF^* (\mu')^* \om$ vanishes along $E$ at order
$$
(k-1) + \ord_E (\hF^* K\mu') + \ord_E( \hF^* (\mu')^* C)~.
$$
On the other hand, $\Om = \mu^* F^* \om$ hence the order of vanishing of $\Om$ along $E$ also equals:
$$
\ord_E(\mu^* F^{-1}(C)) + \ord_E(K\mu) + \ord_E (\mu^* JF)
$$
Let us now introduce $b_E, b_{E'}$ the respective generic multiplicities of $E$ and $E'$, so that
$b_E A(\nu_E) = 1 + \ord_E(K\mu)$ (and similarly for $E'$). Note that
$\ord_E(\hF^* K\mu') = (F_*\ord_E) (K\mu') = k \ord_{E'} (K\mu)$.
We infer
\begin{multline*}
(k-1) + k (b_{E'}\, A(\nu_{E'}) -1) + \ord_E ( \hF^*(\mu')^* C) =\\
= \ord_E(\mu^* F^{-1} (C)) + (b_E\, A(\nu_E) -1)
+  \ord_E(\mu^* JF)
~.
\end{multline*}
The two divisors $\hF^* (\mu')^*C$ and $\mu^* F^{-1}(C)$ are equal because they coincide outside the exceptional locus of $\mu$ and are both numerically trivial. 
We conclude by noting that $b_E \times c(F,\nu_E) = \ord_E(\mu^* F^*\fm_Y) = \ord_E ( \hF^* (\mu')^* \fm_Y) = b_{E'} \times k$.
\end{proof}

\noindent \emph{Action of $F$ on Cartier and Weil classes.}
In the whole discussion, we assume $F: X \to X$ is a finite holomorphic germ. 
For any Weil class $Z$, one can define $F_*Z$ as follows.
Pick some good resolution $\mu_0$, and take $\mu$ such that the lift
$\hF \= \mu_0^{-1} \circ F \circ \mu$ is holomorphic. Then we set $(F_*Z)_{\mu} \= \hF_* (Z_{\mu_0})$.
This does not depend on the choice of $\mu$, and the map $F_*$ is easily seen to be continuous
on $W(\fX)$.

In a similar way, the operator $F^*$ is naturally defined on Cartier classes.
If $Z$ is determined in $X_{\mu_0}$, and $\mu$ is a good resolution such that
$\hF\= \mu_0^{-1} \circ F \circ \mu$ is regular, then we set $F^*Z$ to be the Cartier divisor determined
in $X_\mu$ by $\hF^* Z(\mu_0)$. 

One can show $F^*$ extends continuously to $W(\fX)$, and that $F_*$ preserves the space of Cartier divisors, see~\cite{deggrowth}. For any pair $Z,W$ with one divisor being Cartier, one then
has the following two identities: 
\begin{eqnarray}
F^*Z \cdot W &=& Z \cdot F_* W ~,\label{e:pp}\\
F^* Z \cdot F^*W &=& e \times (Z \cdot W) ~.\label{e:dual2}
\end{eqnarray}
The action on the dual divisors is given as follows.
\begin{lem}\label{l:action-dual}
Pick any valuation $\nu \in \cVqm$. Then 
$F_* Z_\nu = Z_{F_*\nu}$; and $F^*Z_\nu = \sum_{F_\bullet \nu_k = \nu} a_k Z_{\nu_k}$
for some real numbers $a_k >0$.
\end{lem}
Note that the right hand side of the second equation is a finite sum thanks to Lemma~\ref{l:finite}.
\begin{proof}
By continuity it is sufficient to treat the divisorial case. The first equation is a direct consequence 
of~\eqref{e:pp}.  For the second equation, pick models $X_\mu$, $X_{\mu'}$ 
such that $\hF \= (\mu')^{-1} \circ F \circ \mu$ is holomorphic and $Z_\nu$ is determined in $X_{\mu'}$, i.e.
the center of $\nu$ in $X_{\mu'}$ is a prime divisor $E'$.
For any prime divisor $E\in X_\mu$, let $a_E$ be the topological degree of the restriction map $\hF : E \to E'$. By convention we set $a_E =0$ if $\hF(E) \not= E'$.
Then $F^*Z_{E'} = \sum a_E Z_E$ from which one easily infers the second statement of the lemma.
\end{proof}

\subsection{Special singularities}\label{s:special}
Let us begin with the following three classical definitions.
\begin{defi}
A normal surface singularity is called a \emph{simple elliptic} singularity
if the exceptional set of its minimal desingularization is a smooth elliptic curve.
\end{defi}
\begin{defi}
A normal surface singularity is called a \emph{cusp} singularity if 
the exceptional set of its minimal desingularization $\mu: \hX \to X$ is a cycle of rational curves.
This means: either $\mu^{-1}(0)$ is an irreducible  rational curve of self-intersection $\le -1$ with one nodal singularity; or $\mu^{-1}(0)$ is reducible, all components are smooth rational curves, and the dual graph of  $\mu^{-1}(0)$ is a polygon.
\end{defi}

\begin{defi}
A normal surface singularity $(X,0)$ is said to be Kawa\-mata-log-terminal (klt in short) if
$A(\nu) > 0$ for all $\nu \in \hcV_X$.
It is log-canonical (lc in short) if $A(\nu) \ge 0$ for all $\nu \in \hcV_X$.
\end{defi}
Thanks to Corollary~\ref{cor:thin}, it is sufficient to check the condition $A>0$ (or $A\ge 0$) on all edges
of the minimal resolution. In particular $(X,0)$ is klt iff there exists $A_0>0$ such that $A(\nu) \ge A_0$ for all $\nu \in \cV_X$.

Unwinding the definition of $A$ for divisorial valuations in terms of the relative canonical divisor, one sees that the two notions of klt and lc singularities coincide with the standard ones. It is well-known 
that klt and lc singularities in dimension $2$ can be characterized geometrically in a simple way, see~\cite[Theorem~9.6]{kawamata} for a proof.

\newpage

\begin{thm}
~
\begin{itemize}
 \item 
A normal surface singularity $(X,0)$ is klt iff it is a quotient singularity, i.e. it is locally
 isomorphic to the quotient of $\C^2$ by a finite subgroup of $\mathrm{GL}(2,\C)$.
\item
 A normal surface singularity $(X,0)$ is lc but not klt iff 
it is either simple elliptic, a cusp, or a finite quotient of these.
\end{itemize}
\end{thm}
Let us comment briefly on the proof.
One proves that a lc singularity is either rational or simple elliptic or a cusp,~\cite[Lemma~9.3]{kawamata}.
When it is rational, then it is  $\Q$-factorial, hence also $\Q$-Gorenstein. One can then construct a finite cyclic cover of $X$ which is lc and Gorenstein (which is referred to as the log-canonical cover). One concludes using the fact that rational Gorenstein singularities are quotients of $\C^2$ by finite subgroups of $\mathrm{SL}(2,\C)$.
The following can be extracted from the approach to the classification of lc singularities explained
in~\cite{alexeev}, see also~\cite[chapter~4]{kollar-mori}.
\begin{prop}\label{p:class}
Suppose $(X,0)$ is a lc singularity. Let
$\mu: \hX \to X$ be the  minimal desingularization of $(X,0)$, and set $\cA_0 \= \{ \nu \in \Gamma(\mu)\, , A(\nu) =0 \}$. Then the following holds.
\begin{enumerate}
\item When $(X,0)$ is klt, the set $\cA_0$ is empty.
\item When $(X,0)$ is simple elliptic or a cusp, then $\cA_0 = \Gamma(\mu)$.
\item In all other cases, $\Gamma(\mu)$ is a tree. It has either one branched point in which case
$\cA_0$ is reduced to this branched point; or two branched points of valence $3$ and $\cA_0$ is the segment joining them.
\end{enumerate}
\end{prop}
In Kawamata's terminology, a divisorial valuation in $\cA_0$ is called a log-canonical place, see~\cite{kawamata2}.

\section{Holomorphic maps on singular surfaces}\label{s:holo}

\subsection{Proof of Theorem~B}
Let $F: X \to Y$ be any finite holomorphic germ between normal surface singularities.
If $(X,0)$ is klt, then~\eqref{e:jac} gives $A(F_*\nu) \ge A(\nu) > 0$ for any 
valuation $\nu\in\hcV_X$. By Lemma~\ref{l:finite}, $F_*: \hcV_X \to \hcV_Y$ is surjective hence
$(Y,0)$ is klt. The same proof works for lc singularities. 

Suppose now $F: X \to X$ is a finite holomorphic self-map of a normal surface singularity.
Denote by $A_0$ the minimum of $A$ on $\cV_X$, and by $N_0$ the minimum of 
$\nu(JF)$ on $\cV_X$.

If $JF$ is not zero, then $N_0 >0$ by Proposition~\ref{p:maps}.
Pick a valuation $\nu \in \cV_X$, and write $F^n_*\nu = c(F^n,\nu) \times F^n_\bullet \nu$
with $F^n_\bullet \nu \in \cV_X$. Recall that $c(F^n,\nu) = \nu (F^{n*} \fm) \ge 1$.
Whence $A(F^n_*\nu) =  A(\nu) + \sum_{k=0}^{n-1} F^k_*\nu(JF) \ge A_0 + n N_0 >0$
for $n > -A_0/N_0$. Again $F$ being finite, then $F_\bullet: \cV_X \to \cV_X$ is surjective hence
$A>0$ on $\cV_X$. This shows $(X,0)$ is klt.

\smallskip

Now assume $JF\equiv 0$. Set $\cA_0 \=  \{\nu\in \cV_X, \, A(\nu) = A_0 \}$. We shall show that when $A_0 \neq 0$ then $F$ is invertible.
 
Suppose first that $(X,0)$ is klt. Then $(X,0)$ is a quotient singularity, hence there exists a finite map $\pi: (\C^2,0) \to (X,0)$ unramified outside $0$. 
Both maps $\pi$ and $F \circ \pi$ induce unramified maps from a punctured small ball around $0$ in $\C^2$ onto a punctured neighborhood $U^*$ of $0$ in $X$. They thus both define the universal cover of $U^*$. Whence $\pi$ and $F\circ \pi$ have the same degree, which implies $F$ to be invertible.

For the rest of the proof we assume that  $A_0 < 0$. 
For any $n\ge0$, define $\cA_n \= F_\bullet^{-n}(\cA_0)$. The Jacobian formula and the property $c(F,\cdot) \ge1$ implies $\cA_1 \subset \cA_0$ and $c(F,\nu) =1$ for all $\nu\in \cA_1$.

Fix $\nu_0 \in \cA_0$. By Proposition~\ref{p:depend} there exist two constants $c_0, c_1 >0$ such that 
$c_0 Z_{\nu_0} \le Z_\nu \le c_1 Z_{\nu_0}$.
Pick $n\ge 1$ and $\nu \in \cA_n$. Then we can 
write $F^{n*}Z_\nu = \sum a_k Z_{\nu_k}$ for some $\nu_k \in \cA_{2n}$ and $a_k >0$ 
by Lemma~\ref{l:action-dual}. 
But $F^n_* F^{n*} Z_\nu = e^n Z_\nu$ and $c(F^n,\nu_k) = 1$ for all $k$ hence
$e^n = \sum a_k$. Now we have
\begin{multline*}
c_1 (Z_{\nu_0}^2) e^n
\ge
(\sum a_k Z_{\nu_k}) \cdot Z_{\nu_0}
= F^{n*} Z_\nu \cdot Z_{\nu_0}
\\
\ge
c_0 c_1^{-1}\, (F^{n*} Z_{\nu_0} \cdot Z_\nu)
=
c_0 c_1^{-1}\,c(F^n, \nu)\, (Z_{\nu_0} \cdot F^n_\bullet Z_\nu)
\ge
c_0^2 c_1^{-1}\, (Z_{\nu_0}^2)
\end{multline*}
Note that $Z_{\nu_0}^2$ is negative (see Lemma~\ref{l:neg}). Letting $n\to\infty$, we get $e=1$.

%

\subsection{Proof of Corollary~C}
We shall rely on the following result.
\begin{prop}\label{p:lift}
 Suppose $F: (X,0) \to (X,0)$ is a finite germ of topological degree $e \ge2$.
Suppose $(X,0)$ is lc but is not a cusp. Then there exists a proper modification
$\bar{\mu}: \bar{X} \to X$ such that $\bar{X}$ has only klt singularities and $F$ lifts to $\bar{X}$ as a holomorphic map.
\end{prop}
The result is not true in the cusp case as shown in Proposition~\ref{p:cusp} below.
\begin{proof}
Let $\mu: \hX \to X$ be the minimal desingularization of $(X,0)$. If $(X,0)$ is klt there is nothing to prove.
If $\mu^{-1}(0)$ is an elliptic curve, then the lift $\mu^{-1} \circ F \circ \mu$ to $\hX$ is holomorphic since the image of an indeterminacy point is a rational curve.
Otherwise, we are in Case~(3) of Proposition~\ref{p:class}.
Recall that $\Gamma(\mu)$ has one or two branched points, and that the segment joining them
coincides with the set $\cA_0 = \{ A =0 \}$. By Theorem~B, one has $JF=0$, and the Jacobian formula implies $\cA_0$ to be totally invariant. Since $F_\bullet$ is continuous and surjective,
$F_\bullet$ fixes a branched point of $\Gamma(\mu)$ or permute two branched points. For sake of simplicity
we shall assume we are in the former case, i.e.
$F_\bullet \nu = \nu$ for a branched point $\nu$ of $\Gamma(\mu)$ with $A(\nu) =0$.
By Lemma~\ref{l:action-dual}, we may write $F_* Z_\nu = c \times Z_\nu$ with $c=c(F,\nu)$.

Write $F^{n*} Z_\nu = \sum a_k Z_{\nu_k}$ with $a_k >0$ and $\nu_k\in F^{-1}_\bullet\{ \nu\}$.
The Jacobian formula shows  $A(\nu_k) =  0$ and $\nu_k$ belongs to $\Gamma(\mu)$.
Pick $c_0, c_1>0$ such that $c_0\, Z_\nu \le Z_{\nu'} \le c_1 \, Z_\nu$ for all $\nu'\in \Gamma(\mu)$.
Then $e^n Z_\nu = F^n_* F^{n*} Z_\nu =(\sum a_k c(F^n,\nu_k) ) \, Z_\nu$
and $c_0 \times c^n \le c(F^n,\nu_k) \le c_1 \times c^n$ for all $k$ and $n$.
On the other hand, $c^n Z_\nu^2 = F^{n*}Z_\nu \cdot Z_\nu$, so that $c_0\, (\sum a_k) Z_\nu^2 \le c^n Z_\nu^2 \le c_1 \, (\sum a_k) Z_\nu^2$.
Letting $n\to\infty$, we conclude that $e = c^2$.
We may now write $(F^* Z_\nu - c Z_\nu)^2 = e Z_\nu^2 - c^2 Z_\nu^2 =0$. This implies $F^*Z_\nu = c Z_\nu$, hence $\nu$ is totally invariant by $F_\bullet$.

We can now apply the same argument to the other branched point of $\Gamma(\mu)$, and we conclude that all branched points of $\Gamma(\mu)$ are totally invariant. We now blow-down all exceptional divisors of $\mu$ except for the ones associated to the branched points of $\Gamma(\mu)$. In this way, one obtains a surface with cyclic quotient singularities on which $F$ is holomorphic.
\end{proof}

\begin{proof}[Proof of Corollary~C]
Suppose $p\in X$ is a singular point. If its orbit is not totally invariant, then  one can find
a smooth point $q \in X$ such that $F: (X,q) \to (X,p)$ is a finite map.
By Theorem~B~(1), $(X,p)$ is klt. We may thus assume that $p$ is
totally invariant by $F$. If the Weil divisor $JF$ is not zero near $p$, then Theorem~B~(3)
implies $(X,p)$ klt. When $JF=0$, then $(X,p)$ is lc by Theorem~B~(4). 
We apply the preceding proposition to conclude.
\end{proof}
~


\subsection{Maps on quotient singularities}\label{s:quotient}
~
\smallskip

\noindent \emph{Finite quotients of $\C^2$.}
Suppose $ X = \C^2 / G$ with $G$ finite and acting linearly on $\C^2$. 
Now pick any morphism $\rho : G \to G$, and a polynomial map $F:  \C^ 2 \to \C^2$ such that 
$F(g\cdot x) = \rho(g) \cdot F(x)$ for all $x\in \C^2$ and all $g \in G$. We moreover assume that $F(0)=0$; and that no curve is contracted to a point. Then $F$ descends as a finite holomorphic map to the quotient space $X$. Its local topological degree at $0\in X$ is equal to the local topological degree of $F$ at $0$.

For any $G$ one can find a finite holomorphic map $F$ of the germ $(\C^2/G,0)$ such that 
in any model dominating $(\C^2/G, 0)$ the map induced by $F$ is not holomorphic. 
One can show that this is the case as soon as the local topological degree is not the square of 
the asymptotic contraction rate as defined in~\cite{eigenval}. A concrete example is given by
  $G = \langle (-x,-y) \rangle$,  and $P(x,y) = ( x^3, y)$.

Note also that one can realize globally these examples on the quotient space $\mathbb{P}^2/G$, see~\cite{cantat}.
A particular instance of this construction is the case of weighted projective spaces.

\smallskip
\noindent \emph{Latt\`es examples.}
There are endomorphisms of abelian surfaces $F: A \to A$ and finite groups $G$ of automorphisms of $A$ such that $A/G$ is rational and $F$ descends to the quotient.

When $G$ is generated by complex reflections, then $A/G$ is a weighted projective space, see~\cite{TY}.
It may happen that $A/G$ is isomorphic to complex projective space\footnote{a classification of these examples is given in~\cite{KTY}}. There are dynamical  characterizations of these examples by Berteloot-Dupont~\cite{berteloot}.

Other examples than weighted projective spaces may appear. This is the case when $A = (\C/\Z[i])^2$, and $G$ is the cyclic group of order $4$ generated by $(z,w) \mapsto (iz, iw)$. Any linear map $(z,w) \mapsto (az+bw, cz+dw)$ with $a,b,c,d \in \N$ and $ ad-bc \neq 0$ induces a holomorphic map on the quotient of topological degree $|ad-bc|^2$.  One can play the same game with $A= (\C/\Z[j])^2$, $j$ a primitive $3$-root of unity, and $G$ of order $3$, $(z,w) \mapsto (jz,jw)$, or of order $6$, $(z,w) \mapsto (-jz, -jw)$.

Another example arises as follows. Take $\zeta$ a primitive $5$-root of unity. Then the ring of integer $\Z[\zeta]$ of the number field $\Q[\zeta]$ embeds as a lattice in $\C^2$ by the morphism $\sigma = (\sigma_1, \sigma_2)$ where $\sigma_1, \sigma_2$ are the two complex embeddings given by $\sigma_i (\zeta) = \zeta^i$, $i=1,2$. The quotient $\C^2/\sigma (\Z[\zeta])$ is an abelian variety, on which a cyclic group of order $5$ acts naturally generated by $(z,w) \mapsto (\zeta z,\zeta ^2 w)$. The quotient variety is  rational, and for any $\a \in \Z[\zeta]$, the map $(z,w) \mapsto (\sigma_1(\a) z , \sigma_2(\a) w)$ descends to the quotient and induces there a holomorphic map of topological degree the norm $N(\a)$.


\subsection{Maps on simple elliptic singularities}\label{s:elliptic}

Pick any smooth elliptic curve $E$, and $\phi: E \to E$ an endomorphism 
whose lift to the universal cover $\simeq \C$ is given by $z \mapsto \a \, z$ with $\a \in \C^*$. 
The topological degree of $\phi$ is then given by $d\= |\a|^2$. 
For any $a \in \Z$, denote by $\pic_a$ the affine space consisting of all line bundles of degree $a$ on $E$.
The map $L \mapsto \phi^* L \otimes (L^{-1})^{\otimes d}$ is an affine map sending $\pic_a$ to $\pic_0$ whose
linear part is given by $\bar{\a} - d$.
Hence for any $a$, there exists a line bundle $L$ on $E$ of degree $a$ and such that $\phi^*L $ is isomorphic to $L^{\otimes d}$. The map $\phi$ hence lifts as a holomorphic map $\hat{\phi}$ from the total space $X(L^{\otimes} d)$ of $L^{\otimes d}$ to $X(L)$, and acts linearly on the fibers. There exists a natural map $X(L) \to X(L^{\otimes d})$ fixing each point of the base $E$, and acting as $ \zeta \mapsto \zeta^d$ in the fibers.
Composing with $\hat{\phi}$ we get a holomorphic map $F: X(L) \to X(L)$. The zero section $E_0$ in $X(L)$ has self-intersection $a$. If $a$ is negative,
we may contract it to a simple elliptic surface singularity $X$, and $F$ descends to a finite holomorphic map on $X$.

Note that we may compactify $X(L)$ by adding a point at infinity in each fiber to get a ruled surface $\bar{X}(L) \to E$, such that $F$ extends as a holomorphic map on this surface.
Its topological degree is equal to the square of its asymptotic degree which is given by $d$, see Proposition~\ref{p:fibr}.

Building on the previous example and keeping the same notation, take $G$ a finite non-trivial group of automorphisms of $E$ commuting with $\phi$ and such that $g^* L \simeq L$ for any $g \in G$.
The group lifts to an action on $X(L)$ and on the surface $X$ obtained after contracting the zero section of $L$ which is linear in the fibers. The map $F$ induces holomorphic maps on the resulting quotient surfaces $X(L)/G$ and $X/G$.
Note that $X/G$ has a single lc singularity, whereas $X(L)/G$ has quotient singularities.
Finally note that the projective surface $\bar{X}(L) / G$ is rational since it is ruled over $E/G \simeq \P^1$.

For any fixed couple $(E,\phi)$ there exists a group of order $2$ satisfying these conditions. If $E \simeq \C/\Z[i]$, $i^2 =-1$,  (resp. $E \simeq \C/\Z[j]$ with $j^3=1$) and any endomorphism on these curves, then there exists a group of order $4$ (resp. $3$ and $6$) satisfying these requirements.


\subsection{Finite maps on cusp singularities}

It is not difficult to see that any cusp singularity admits a non-invertible finite holomorphic self-map.
We aim at proving the following more precise result.
\begin{prop}\label{p:cusp}
There exists a cusp singularity $(X,0)$ and a non-invertible finite holomorphic self-map $F: (X,0) \to (X,0)$ such that for any non-trivial modification $\bar{\mu}: \bar{X}\to X$ with $\bar{X}$ normal, the lift $\bar{F} \= \bar{\mu}^{-1} \circ F \circ \bar{\mu}$ is not holomorphic on $\bar{X}$.
\end{prop}

The rest of this section is devoted to the proof of this proposition. We rely on a description of cusp singularities in arithmetical terms (see~\cite{oda}).

Let $d\ge2$ be square-free integer, and pick $N$ a rank $2$ free $\Z$-module of the quadratic field $\Q(\sqrt{d})$. One may assume $N = \Z + \Z \om$ for some irrational  $\om\in\Q(\sqrt{d})$. Denote by $\al'$ the conjugate of any $\al\in \Q(\sqrt{d})$.
The map $\sigma(\al) = (\al, \al')$ embeds $N$ as a lattice in $\R^2$.
An element $\al \in \Q(\sqrt{d})$ is said to be totally positive if $\al$ and  $\al'$ are both positive. The group of invertible elements $\e$ in the ring of algebraic integers of $\Q(\sqrt{d})$
such that $\e N = N$ is denoted by $U_N$, and its subset of totally positive elements  by $U^+_N$. By Dirichlet's theorem, one has $U^+_N\simeq \Z$, and $[U_N:U^+_N] = 1$ or $2$ depending on whether $-1$ belongs to $U_N$ or not.

The quotient of $\C^2$ by the action of $N$ given by $\C^2 \ni (z,z') \mapsto (z +n, z'+n')$ is a complex torus $T_N$. If $M$ denotes the set of all $\R$-linear maps $m$ on $\R^2$ such that $m(N) \subset \Z$, the functions $e(m) (z,z') \= \exp( 2i\pi \times m(z,z'))$ identify $T_N$ canonically with the toric variety defined by the trivial fan $(0)$ in $N$. Write $\pi: \C^2 \to T_N$ for the natural quotient map.

Fix a generator $\e$ of $U^+_N$. The map $\phi_\e(z,z') = (\e z, \e z')$ acts on $\C^2$ by preserving the product of the upper half planes $\H\times\H$. This action descends to $H_N \= \pi ( \H^2)\subset T_N$.
It is a fact that this action is properly discontinuous, that $X_N \= H_N/\langle \phi_\e \rangle$ is an open smooth surface to which one can add $[U_N:U^+_N]$ points to obtain a compact surface with cusp singularities. In order to explain the construction of the required endomorphism, we need to explain some aspects of the proof of this fact. 

Denote by $N_+$ the subset of all totally positive elements in $N$. We identify it by $\sigma$ with $N \cap \R_+^2$. Denote by $\Theta$ the convex hull of $N_+$. Its boundary consists of a broken line
which is the union of segments joining extremal points of $\Theta$ that we denote by 
$n_k$, $k \in \Z$. We do the same construction with $N\cap(\R_+ \times \R_-)$, and obtain another convex set $\Theta_*$ whose extremal points will be denoted by $n_s^*$, $s\in \Z$.
We now introduce the non-singular fan $\Delta$ which consists of the two infinite sets of rays $\R_+ n_k$, $\R_+ n_s^*$ for all $k,s$, and of the cones $\R_+ n_k + \R_+ n_{k+1}$, $\R_+ n_s^* + \R_+ n_{s+1}^*$. The toric variety $T_N(\Delta)$ one obtains in this way contains two infinite chains of rational curves (corresponding respectively to rays $\R_+ n_k$ and $\R_+ n_s^*$). 

The map $\phi_\e$ acts naturally on $N$ preserving $\Theta$ and $\Theta_*$. It induces a "translation" on the set of rays such that $\phi_\e (n_k) = n_{k+l}$, and $\phi_\e (n_s^*) = n_{s+t}$ for some $l,t \in \N^*$ and all $k,s \in \Z$. Using this, one can show that the action of $\phi_\e$ on $T_N(\Delta)$ is properly discontinuous and cocompact. The quotient space $\bar{X}_{N}$ is a smooth compact complex surface which contains $X_N$ as a dense open subset. Its complement is 
the quotient by $\phi_\e$ of the two chains of rational curves associated to the rays
 $\R_+ n_k$ and $\R_+ n_s^*$. It thus produces two cycles of rational curves on $\bar{X}_N$.
These cycles can be contracted yielding  a new singular surface
 $\bar{X}^*_N$ with two cusp singularities $p$ and $p_*$.
 
 We denote by $\Gamma$ and $\Gamma^*$ the  dual graphs of the minimal desingularizations of $p$ and $p_*$ 
respectively. The graph $\Gamma$ is naturally  the set of rays $\R_+^*\cdot (v,v') \in \R_+^2$  modulo the action of $\phi_\e$ given by $\phi_\e(\R_+^* \cdot (v,v')) = \R_+^* \cdot(\e v, \e v')$. One may thus identify $\R/\Z \simeq \Gamma$ by  mapping $t\in \R$ to $\R_+^* \cdot (1,\e ^{2t})$.
 The same description is valid for $\Gamma^*$.

Let us now fix a totally positive element $\a \in \Q(\sqrt{d})$ such that $\a N \subset N$. Then the map $\phi_\a(z, z') \= (\a z, \a' z')$ on $\C^2$ preserves $\H^2$ and  commutes with $\phi_\e$.
It thus induces a holomorphic map $F_\a$ on $X_N$, which extends to $\bar{X}^*_N$ by normality.
The  topological degree of $F_\a$ is the norm of $\a$, that is $N(\a)\= \a \a'$.

The map $\phi_\a$ also acts on $\R_+^2$ hence on the set of rays in the upper half quadrant. 
The induced action on the quotient of the set of rays modulo $\phi_\e$ gives us precisely the induced action of $F_\a$ on the dual graph $\Gamma$ through the identification above.
Concretely in the parameterization $\R/\Z \ni t \mapsto \R_+^* \cdot (1,\e ^{2t})$, we get 
$$F_\a (t) = t + \frac{\log (\a/\a')}{2 \log \e}~.$$  
The action is thus a rotation of angle $\log (\a/\a')/(2 \log \e)$.
 
 When this angle is a rational number (this is the case if $\a$ is chosen to be an integer), then $F_\a$
 lifts to a holomorphic map of some model of $\bar{X}^*_N$ with at most quotient singularities.
 Otherwise, $F_\a$ can not be lifted as a holomorphic map in any model dominating $\bar{X}^*_N$. A detailed 
argument of this fact  is given in a similar context in~\cite{monomial}.

 To conclude let us exhibit an example where the rotation is irrational. Take $d=2$, and $N$ to be the
ring of algebraic integers of $\Q(\sqrt{2})$, that is
$\fD = \Z + \sqrt{2} \Z$. A totally positive fundamental unit is given by $\e = 3 + 2 \sqrt{2}$.
Pick $\a  = 3 + \sqrt{2}$ which is also totally positive. Then $\a/\a' = \frac17 (11 + 6\sqrt{2})$.
This number is not integral, hence none of its powers are integral.
None of them are hence a power of $\e$, and $\log (\a/\a')/\log \e$ is irrational as required.

\section{Numerical characterization of holomorphic maps}\label{s:numeric}

In this section, we prove Theorem~A. The proof we give is based on the analysis of the action of $F$ on the N\'eron-Severi space of all birational models of $\mathbb{P}^2$. We thus need the analogous global construction of Weil and Cartier divisors that we described  in a local setting in Section~\ref{s:singular}.
We only sketch here the theory referring to~\cite{deggrowth} for details.

%
%
\subsection{Global Weil and Cartier classes}\label{s:WC}
A model of $\mathbb{P}^2$ is an equivalence class of bimeromorphic map $\pi: X_\pi \dashrightarrow \mathbb{P}^2$ where $X_\pi$ is a normal surface. The set of all models forms an inductive set. One defines $\fP \= \varprojlim_\pi X_\pi$, and endow this space with the product topology, each $X_\pi$ viewed as a scheme with its Zariski topology. For each $X_\pi$ smooth, we let $\NS(X_\pi)$ be its N\'eron-Severi space. For $X_\pi$ singular, $\NS(X_\pi)$ is defined to be the orthogonal complement in $\NS(X)$
with respect to the intersection form of the space generated by classes associated to exceptional divisors of some desingularization of $X \to X_\pi$. For any bimeromorphic regular map $X_{\pi'} \to X_\pi$, we have an orthogonal decomposition:
\begin{equation}\label{e:small-decomp}
\NS(X_{\pi'}) = \NS(X_\pi) \oplus \Z [E_i]
\end{equation}
where the sum is taken over all irreducible prime divisors $E_i$ contracted by $\pi^{-1} \circ \pi'$.
A Weil class is an element $Z\in \varprojlim_\pi \NS(X_\pi)$.  Equivalently, it is a collection $Z = \{ Z_{X_\pi} \}$ with $Z_{X_\pi} \in \NS(X_\pi)$ with the compatibility condition $ (\pi^{-1} \circ \pi')_* Z_{X_\pi'} = Z_\pi$ whenever $\pi^{-1} \circ \pi': X_{\pi'} \to X_{\pi}$ is regular. The class $Z_{X_\pi}$ is called the \emph{incarnation}  of $Z$ in $X_\pi$.
A Cartier class is an element $Z\in \varinjlim_\pi \NS(X_\pi)$. As in the local case, we may speak of a model determining a Cartier class.
The infinite dimensional space of Weil (resp. Cartier) classes is denoted by $\NS(\fP)$ (resp. by $\CNS(\fP)$).

The intersection pairing on each $\NS(X_\pi)$ induces a pairing $\CNS(\fP) \times \NS(\fP) \to \R$ whose restriction to $\CNS(\fP)$ is a symmetric bilinear form of Minkowski's type.  One can define the completion of $\CNS(\fP)$ with respect to this intersection form. This yields a space $\Ltwo(\fP)$ on which the intersection form is well-defined, and which is complete for  the Hilbert norm: $\| Z \|^2 \= - Z^2 + 2\, (Z\cdot \cL)^2$. Here $\cL$ is the Cartier class determined in $\mathbb{P}^2$ by $\cO(1)$.

For any fixed model $X \dashrightarrow \mathbb{P}^2$, we introduce  the set $\mV_X$ of all divisorial valuations on $\mathbb{C}(X)$ (up to equivalence) whose center in $X$ is a point. Each  $\nu \in \mV_X$  determines a Cartier divisor $E_\nu \in \CNS(\fP)$ as follows. Pick a  desingularization  $\mu: X' \to X$ such that $\nu$ has one dimensional center in $X'$, and $\pi' : X' \to \mathbb{P}^2$ is regular.
Contract all $\mu$-exceptional curves except for the center of $\nu$. You obtain in this way an intermediate model $X' \to X_0 \to X$. Then define $E_\nu$ to be the Cartier class determined in $X_0$ by the center of $\nu$.
Then we have the following orthogonal decompositions:
\begin{eqnarray}
\CNS(\fP) & = & \NS(X) \oplus_{\nu \in \mV_X} \R E_\nu  \label{e:ortho1} \\
\NS(\fP) & = & \NS(X) \oplus \mV_X^{\R} \label{e:ortho2}\\
\Ltwo (\fP) & = & \NS(X) \oplus \Ltwo(\mV_X)~, \label{e:ortho3}
\end{eqnarray}
where $\Ltwo(\mV_X)$ is the set of real-valued functions $f: \mV_X\to \R$ such that
$\sum_{\mV_X} f^2(\nu)$ is summable.

A Weil class $Z$ is pseudo-effective if its incarnations $Z_X\in\NS(X)$ are pseudo-effective for all models $X$. A Weil class is nef if $Z_X$ is nef for all $X$. This implies $Z\in\Ltwo(\fP)$ and $Z^2 \ge 0$.

\medskip

Take $F: \mathbb{P}^2 \dashrightarrow \mathbb{P}^2$ a dominant rational map of topological degree $e$. Then we have induced actions $F_*, F^*$ on $\CNS(\fP), \NS(\fP)$ and $\Ltwo(\fP)$ such that
\begin{eqnarray}
F_*F^* Z & = & e \times Z ~, \label{e:pp2}\\
F^* Z \cdot Z' & = & Z \cdot F_* Z'~, \label{e:trans}
\end{eqnarray}
for all $Z,Z' \in \Ltwo(\fP)$.
If $X$ is an arbitrary model, we denote by $F_\#,F^\#: \NS(X) \to \NS(X)$, the composition maps $F_\# \= \mathrm{p}_X\circ F_*$, $F^\# \= \mathrm{p}_X \circ F^*$ where $\mathrm{p}_X : \NS(\fP) \to \NS(X)$ is the orthogonal projection given by~\eqref{e:ortho2}. An alternative description of these maps can be given as follows. Let $\hX$ be the graph of $F$ in $X\times X$, and denote by $\pi_1, \pi_2$ the projection onto the first and second factor respectively so that $F = \pi_2 \circ \pi_1^{-1}$. Then $F_\# \= \pi_{2*} \circ \pi_1^*$ and $F^\# \= \pi_{1*} \circ \pi_2^*$.
For these linear operators,~\eqref{e:pp2} is no longer true but one still has:
\begin{eqnarray}
F^\# Z \cdot Z' & =&  Z \cdot F_\# Z'~,
\end{eqnarray}
for all $Z,Z'\in\NS(X)$.

The degree of $F$ is defined as the degree of the preimage of a generic line, and is equal to $\deg(F) = F^* \cL \cdot \cL$. The sequence $\deg(F^n)$ is sub-multiplicative, and one sets $\la \= \lim_n \deg(F^n)^{1/n}$.
In~\cite{deggrowth}, the following theorem is proved.
\begin{thm}\label{t:funda}
Suppose $e< \la^2$. Then there exist two nef classes $\theta^*, \theta_*$ such that 
$F^* \theta^* = \la \theta^*$, $F_* \theta_* = \la \theta_*$,  and
for all classes $\a \in \Ltwo(\fX)$:
$$
\la^{-n} F^{n*} \al \to (\al\cdot \theta_*) \theta^* ; \,
\la^{-n} F^{n}_* \al \to (\al\cdot \theta^*) \theta_*  
~.$$
\end{thm}
Using the invariance properties of $\theta^*$ and $\theta_*$ it is not difficult to show that $(\theta^*)^2 =0$  and $\theta^* \cdot \theta_* >0$.

We also gather here two computations of general interest.
\begin{lem}
Suppose $\theta_*, \theta^*$ are Cartier classes determined in a smooth model $X$.
Denote by $K_X$ a canonical divisor of $X$ and by $JF$ the critical divisor of $F$.
Then one has:
\begin{eqnarray}
 (e- \la) \, K_X \cdot \theta^* & = &  - \la \, (\theta^* \cdot JF)~;  \label{e:crit+}
\\
(\la -1) \, K_X \cdot \theta_* & = & -\theta_* \cdot JF~. \label{e-crit-}
\end{eqnarray}
\end{lem}
\begin{proof}
Pick a sequence of point blow-ups $\pi: Y \to X$ such that $F = G \circ \pi^{-1}$ with $G: Y \to X$ holomorphic. The proof is then a consequence of the invariance properties of $\theta_*$ and $\theta^*$, of the Jacobian formula $G^* K_X = K_Y - JG$,  and of the relation $\pi_* (JG) = JF$.
\end{proof}

%
%

\subsection{Holomorphic model}

The following proposition gives one implication in Theorem~A.
\begin{prop}\label{p:cartiert*}
Assume $e< \la^2$. Suppose there exists a model $X$ in which $F$ induces a holomorphic map.
Then both classes $\theta_*$ and $\theta^*$ are Cartier and determined in $X$.
\end{prop}
\begin{rmk}
One cannot drop the assumption $e < \la^2$. Take $F(x,y) = (y^2, x^3)$ so that $e =6$, $\la = \sqrt{6}$.
There exists a class $\theta_*$ such that $F_* \theta_* = \la \theta_*$.
However this class is not Cartier, as
it is  associated to the \emph{irrational} monomial
valuation $\nu(x) =-1$, $\nu(y) = -\sqrt{2/3}$ in the sense
$\theta_* \cdot Z(P) =  - \nu(P)$ for any polynomial $P \in \C[x,y]$. Here  $Z(P)$ is the Weil divisor
whose incarnation in a model $\pi: X \to \mathbb{P}^2$ is given by
the part of the divisor of the meromorphic function induced by $P$ on $X$ which lies above the line at infinity.
\end{rmk}
\begin{proof}
We claim that $F$ cannot contract any curve in $X$.
Indeed $\sqrt{e}^{-1} F^\#$ is a linear map on $\NS(X)$ which preserves the intersection form since $F$ is holomorphic. Therefore, it is bijective. By duality, $F_\#$ is also injective hence $F_\# C \neq 0$ for any irreducible curve $C$. This proves our claim.

The linear map $F^*$ preserves $\NS(X)$ as the pull-back of a Cartier class determined in $X$ is by definition determined in $X$ since $F$ is holomorphic. 
The map $F_*$ also preserves $\NS(X)$ since $F$ does not contract curves, see~\cite[Corollary~2.6]{deggrowth}.

Now pick any $\theta \in \NS(X)$ with $\theta^2 >0$ so that $\theta \cdot \theta^*, \theta\cdot \theta_* >0$. 
Then by Theorem~\ref{t:funda}, $\la^{-n} F^{n*}\theta \to (\theta\cdot\theta_*) \theta^*$ and 
$\la^{-n} F^{n}_*\theta \to (\theta\cdot\theta^*) \theta_*$. This shows $\theta^*,\theta_*\in\NS(X)$ and concludes the proof.
\end{proof}


\subsection{When are the classes $\theta^*, \theta_*$ Cartier?}
To conclude to the proof of Theorem~A, we need to prove that if $\theta^*$ and $\theta_*$ are Cartier then $F$ admits a holomorphic model or is a skew product. 
We first prove an intermediate result. 
\begin{thm}\label{t:contract}
Suppose $e< \la^2$ and $\theta^*, \theta_*$ are Cartier.
Then we are in one of the following two exclusive cases.
\begin{itemize}
 \item The set of all curves orthogonal to $\theta^*$ can be contracted, and $F$ induces a 
holomorphic map on the resulting normal surface.
\item There exists a model $X$ and a fibration $\varpi: X \to \mathbb{P}^1$ which is preserved by $F$ and 
such that the class of a fiber equals $\theta^*$.
\end{itemize}
\end{thm}

\begin{rmk}
This result  is proved in~\cite{DDG} in the case $e<\la$. Under this assumption, the second case never appears. 
To deal with the case $e \ge \la$, we rely on an argument in~\cite{DJS,cantat} together with the classification of rational surfaces with special anti-canonical divisors given in~\cite{sakai}.
\end{rmk}

\begin{rmk}
Let us construct an example where the fibration induced by $\theta^*$ is neither elliptic nor rational. Pick any smooth curve $C$ of genus $g\ge 2$. Take $L$ a line bundle of degree $0$ such that $L^{\otimes d}$ is isomorphic to $L$ for some integer $d\ge 2$. Then as in Section~\ref{s:elliptic}, on  the compactified total space $\bar{X}(L)$, the composition map $F: \bar{X}(L) \to \bar{X}(L^{\otimes d}) \to \bar{X}(L)$ 
is holomorphic. It preserves each fiber of the natural fibration $\bar{X}(L) \to C$, and leaves
the zero section $C_0$ totally invariant. The topological degree and the asymptotic degree of $F$ are both equal to  $d$. Hence $\theta^* = [C_0]$. One can also take finite quotients of these examples. 
\end{rmk}

\begin{proof}
Let $X$ be a smooth model of $\mathbb{P}^2$ in which both  classes
$\theta^*$ and $\theta_*$ are determined. Pick $Y$ a desingularization of the graph of $F$ in 
$X\times X$ and denote by $\pi: Y \to X$ and $G: Y \to X$ the projections onto  the first and second factor respectively. Then $\pi$ is a composition of point blow-ups, $G$ is a holomorphic map, and $F = G \circ \pi^{-1}$.

We let $H^*$ be the  
sub-lattice of $\NS(X)$ generated by classes of irreducible curves $C$ such that 
$C \cdot \theta^* = 0$, and denote by $H^*_\R$ the real vector space $H^* \otimes \R \subset 
\NS_\R(X)$.

\medskip

\noindent {\bf 1.} Suppose first that $\theta^* \notin H^*_\R$. Then the intersection form on $H^*$ is negative definite so that we may contract all curves $C$ whose class in $\NS(X)$ belongs to $H^*$. We get a holomorphic modification $\varpi: X \to X_0$ where $X_0$ is a compact normal analytic space.
We claim that $F$ induces a holomorphic map on $X_0$. If it were not the case, then $F$ would admit a point of indeterminacy $p$. We could then find a curve $D\subset Y$ such that $\varpi \pi(D) =p$ and
$G(D) = C$ is an irreducible curve which is not contracted  by $\varpi$.
By construction, $\pi(D)$ is contracted by $\varpi$ hence $D \cdot \theta^* =0$, and
the image of $D$ by $G$ is $C$. Being a Cartier class determined in $X$, the pull-back of $\theta^*$ is determined in $Y$ by $G^* \theta^*_X$, where $\theta^*_X$ is the incarnation of $\theta^*$ in $X$. From the equation of invariance
$F^* \theta^* = \la \theta^*$, we get $G^* \theta^*_X = \la \pi^* \theta^*_X$.
Intersecting with $D$, we conclude
$$
\theta^*_X \cdot G_*(D) = G^*\theta^*_X \cdot D
= \la \, \pi^*\theta^*_X \cdot D
= \la \,  \theta^*_X \cdot \pi_*D
=0~.
$$
As $G_*(D)$ is proportional to $C$, this shows that the class of $C$ belongs to $H^*$, hence $C$ should be contracted by $\varpi$. This gives the required contradiction.

\medskip

\noindent {\bf 2.} Assume now that $\theta^*$ belongs to $H^*_\R$. Recall that $(\theta^*)^2 =0$.
The intersection form on $H^*$ is integral and its kernel is one-dimensional generated by $\theta^*$.
Multiplying by a suitable rational number, the class $\theta^*$ becomes integral so that one can find a line bundle $L^*$ on $X$ such that
$\theta^* = c_1(L^*)$.

Consider the map $F^\sharp \= \pi _* \circ G^*$: it is a linear map from $\NS_\R(X)$ to itself which preserves the lattice $\NS(X)$. Fix a basis of $\NS(X) = \oplus_1^h \Z \alpha_i$, $h = \dim_\R \NS_\R(X)$, and write $\theta^*_X = \sum_1^h t_i \alpha_i$ with $t_i \in \Z$.  We have $G^* \theta^*_X = \la \pi^* \theta^*_X$, hence $F^\sharp \theta^*_X = \la \theta^*_X$. This implies for all $i$:
$\la t_i  \in \sum_1^h \Z t_j$. We conclude that $\la$ belongs to $\frac1q \Z$ for some fixed integer $q$
(which can be taken as any non-zero $t_i$).
The same argument applied to the $n$-th iterate of $F$ yields $\la^n \in \frac1q \Z$ for all $n$. This implies $\la$ is an integer.

By looking at the map $F_\sharp \= G_* \circ \pi^*$, and using $F_* \theta^* = (e/\la)\, \theta^*$, we get
$e/\la \in \N^*$ by the same argument.

\medskip

 {\bf 2a.} Suppose now that the set of irreducible curves $C$ in $X$ such that 
$C\cdot \theta^*=0$ is infinite. 
Pick an infinite sequence of such curves $\{C_i\}_{i \in \N}$. As $\NS_\R(X)$ is finite dimensional, one can find a non-trivial relation $\sum_I a_i C_i = \sum_J b_j C_j$, $I,J \subset \N$. We may assume all coefficients are positive integers and $I\cap J = \emptyset$. The class $\theta \= \sum_I a_i C_i$ is then a nef class and satisfies $\theta \cdot \theta^* =0$. It is thus proportional to $\theta^* = c_1(L^*)$. This shows that some power of $L^*$ admits at least two independent sections. As $(\theta^*)^2=0$, the linear system $|nL^*|$ for $n$ large enough induces a fibration $ \varpi : X \to \mathbb{P}^1$. As $F^*\theta^* = \la \theta^*$,  the map $F$ preserves this fibration.

To complete the proof, we have to show now that under our assumption $e<\la^2$, case 1 and 2a above are the only possible ones. From now on, we thus seek a contradiction.

{\bf 2b.} We are in the following situation. The set of irreducible curves $C$ in $X$ such that $C\cdot \theta^*=0$ is finite, 
and their classes generate a subspace in $\NS(X)$ which contains $\theta^*$. Recall that $\theta^* = c_1(L^*)$ for some line bundle $L^*$. We claim that $h^0(nL^*) =1$ for all $n\ge0$.

It is clear that $h^0(nL^*) \le 1$ by the finiteness of curves orthogonal to $\theta^*$. On the other hand, 
we may write $L^*= C_+ - C_-$ with $C_+ = \sum a_i C_i$, $C_- = \sum b_j C'_j$, $a_i, b_j \ge 0$, $C_i \cdot \theta^* = C'_j \cdot \theta^* =0$, and $(\cup C_i) \cap (\cup C'_j)$ is finite. If $C_-$ is non zero,
then $(C_-)^2= C_+ \cdot C_- \ge 0$, so $C_+ =0$ by Hodge. This is not possible as $\theta^*$ is nef.
Note that the same argument shows the support $V$ of the unique section of $L^*$ is connected.

 In the sequel, we write $L^* = \sum a_i C_i$ with $a_i \ge 1$, and denote by $C'_i$ the irreducible
curves  orthogonal to $\theta^*$ and not in $V\= \cup C_i$.
We shall also assume $L^*$ is non divisible in $\Pic(X)$. Finally  recall that we know that $e/\la$ is an integer, so that $e-\la \ge 0$.

\smallskip
%

{\bf 2b1.} Suppose first  $\theta^* \cdot K_X \le 0$.
By Riemann-Roch, we get
$h^0(nL^*) \ge 1 - \frac{n}2 L^* \cdot K_X$ for all $n$, so that $L^* \cdot K_X =0$. Again by Riemann-Roch:
$$
h^0(L^*+ K_X) + h^2(L^*+K_X) \ge 1 + \frac12 \left( (L^*+K_X)^2 - K_X\, (L^*+K_X) \right) =1 ~.
$$
By Serre's duality $h^2(L^*+K_X) = h^0(-L^*) =0$, hence $L^*+K_X$ is effective. As $L^* \cdot (L^* + K_X) = 0$, we may  decompose the canonical divisor as follows:
$K_X = \sum b_i C_i+ \sum c_j C'_j$ with $ b_i,c_j   \in \Z$, $a_i \ge - b_i$, and $c_j \ge 0$.
Note that since $L^*\cdot (L^*+K_X) =0$ the $C'_j$'s have disjoint support from $V$.
The intersection form on the lattice generated by the $C'_j$'s is negative definite, hence
$\sum c_j C'_j \cdot K_X = (\sum c_j C'_j)^2 < 0 $. In particular there exists at least one index $j$ such that $C'_j \cdot K_X <0$. By adjunction, it is necessarily an exceptional $(-1)$ rational curve that we may contract. In doing so, we get a smooth rational surface in which $\theta^*$ is still determined since $\theta^* \cdot C'_j =0$. After performing finitely many such contractions, we cancel all divisors $C'_j$'s and $K_X = \sum b_i C_i$.

We now proceed by contracting all divisors among the $C_i$'s that intersect negatively $K_X$. We arrive in a situation where $K_X\cdot C_i \ge 0$ for all $i$. Since $L^* = \sum a_i C_i$ with $a_i>0$ and $L^*\cdot K_X =0$ we get $K_X \cdot C_i =0$ for all $i$, whence $K_X^2 = K_X \cdot \sum b_i C_i = 0$.
Hodge index theorem then implies that $L^*$ and $K_X$ are proportional, and we may assume $L^* = -K_X$.
By~\eqref{e:crit+}, this implies the critical set of $F$ is included in $V$ (the support of $L^*$). Hence the restriction of $F$ to the complement of $V$ in $X$
is holomorphic and non ramified. Pick a desingularization $\Gamma$ of the graph of $F$. We have birational morphism $p: \Gamma \to X$, and a holomorphic map $G: \Gamma \to X$ such that $F = G \circ p^{-1}$.
Since $F^*\theta^* = \la \theta^*$ we have $G^* L^* = \la p^* L^*$ which implies $G^{-1} (V) = p^{-1}(V)$.
In particular, the restriction of $F$ to $X\setminus V$ is necessarily proper.
We conclude that $F$ induces a finite covering map from $X\setminus V$ onto itself.

Rational surfaces such that $-K_X$ is effective and  $K_X \cdot E =0$ for all curves $E$ in the support of the anticanonical divisor are called generalized surfaces of Halphen type. They were all classified by Sakai in~\cite{sakai}.  To conclude the proof, we note that the fundamental group of the complement of the support of $L^*$ in $X$ is finite, see Proposition~\ref{p:halphen} below. This implies the topological degree of $F$ to be equal to $1$. Whence $\la = e =1$ which contradicts our working assumption $ e < \la ^2$. 
\begin{prop}\label{p:halphen}
Let $X$ be a generalized surface of Halphen type such that $h^0(-nK_X) =1$ for all $n$, and denote by $V$ an effective anti-canonical divisor. Then $\pi_1(X\setminus V)$ is finite.
\end{prop}
We give a proof of this fact at the end of this section.



%

\medskip

{\bf 2b2.} Finally assume that $K_X\cdot \theta^* >0$. Note that by~\eqref{e:crit+}, $e =\la$. Write $\theta^* = c_1(L^*)$, $L^* = \sum a_i C_i$ with $a_i \ge1$ and $C_i$ irreducible as above, and let  $V\= \cup C_i$ as before.
As $F^* \theta^* = \la \theta^*$ and $nL^*$ has a unique section for all $n$, 
$V$ is totally invariant by $F$.
We now use the fact that $\theta_*$ is Cartier and assume it is determined in $X$. Since $\la = e$ is an integer, and $\la$ is a simple eigenvalue of $F_*$, the class
$\theta_*$ also belongs to $\NS(X)$.
Consider now the unique set of non-negative integers $b_i \ge0$ such that
$\sum b_i \ord_{C_i}(Z) = Z \cdot \theta_*$ for any divisor $Z$ supported on $V$.

Pick a model $\mu: X' \to X$ such that the induced map $F': X \dashrightarrow X'$  by $F$ does not contract any of the
divisors $C_i$. Let  $D_i$ be the (proper) image of $C_i$ by $F$.
Then $F_* \ord_{C_i}$ is proportional to $\ord_{D_i}$.
Pick any Cartier divisor $W$ which is determined in $X'$. Since $F_* \theta_* = \la \theta_*$ in $\NS(\fP)$, we have:
$$(F')^\#W_{X'}\cdot (\theta_*)_{X}= F^*W \cdot \theta_* = \frac{1}{\la} \, W \cdot \theta_*= W_{X'}\cdot (\theta_*)_{X'}$$
and
$\sum b_i (F_*\ord_{C_i})(W) = \frac{1}{\la} \,\sum b_i  \ord_{C_i}(W)$.
If $b_i>0$, the curve $D_i$ is thus one of the $C_j$'s.
We conclude that  $F_*$ permutes the set of divisorial valuations
$\ord_{C_i}$ with $b_i>0$. In particular, replacing $F$ by an iterate, we may suppose all curves with $b_i>0$ are fixed. Again by the invariance property of $\theta_*$, for  a curve $C_i$ with $b_i >0$,
we get that  $F^*[C_i] - \la [C_i]$ has no support on $C_i$.
But  the topological degree of $F$ being equal to $\la$, this means $C_i$ is totally invariant in the following strong sense. Pick a model $\pi: X'\to X$ dominating $X$ such that the map $F' \= F \circ \pi$ is holomorphic.
Then no prime divisor in $X'$ is mapped surjectively onto $C_i$ except for the strict transform of $C_i$ itself. Contract all curves $C_i$ with $b_i =0$. Then
$F$ is  holomorphic in this new model at least along the image of $V$. Contracting curves orthogonal to $\theta^*$ and not in $V$, we get an induced map which is holomorphic everywhere.
To simplify notation we continue to denote by $X$ this model.

Let us now focus on $\theta_*$. Because $\theta_*$ belongs to $\NS(X)$, we have $\theta_* = c_1(L_*)$ for some line bundle $L_*$.
Write $F^*\theta_* = \theta_* + h$ with $h \in \NS(X)$. Then  $e \theta_* = F_* F^*\theta_* = F_*\theta_* + F_* h$, hence $F_* h =0$ as $e=\la$. As in the proof of Proposition~\ref{p:cartiert*}, the map $F_*$ is injective on $\NS(X)$, so that $h=0$. We get $F^*\theta_* = \theta_*$ which implies $\theta_*^2 =0$.
From~\eqref{e-crit-}, we get $\theta_* \cdot K_X = - (\la-1)^{-1} \theta_* \cdot JF$ and
$JF \ge (\la-1) \sum C_i\ge \e \theta^*$ for some small $\e>0$, whence $\theta_*\cdot K_X <0$. Riemann-Roch then implies $h^0(L_*) \ge 2$. We conclude that $L_*$ determines a rational fibration which is preserved by $F$.

Let us go back to the support $V$ of $L^*$.  It is necessarily irreducible, otherwise $e$ would be equal to $\la^2$ (apply~\eqref{e:pp2} at an intersection point for instance and recall that $V$ is connected). As $\la >1$, $F$ is critical along $V$ hence all singular points on $V$ are quotient singularities by Theorem~C~(3). It follows that $V$ is $\Q$-Cartier.  The arithmetic genus of $V$ is greater or equal to $2$ because $K_X \cdot V >0$. Hence  the restriction map $F|_V$ is of finite order, so that we can assume it to be the identity.
For a point $x \in V$, denote by $L_x$ the fiber of the rational fibration lying over $x$, and $F_x$ the induced map on $L_x$. It is a rational self-map of $L_x$ which leaves totally invariant $L_x \cap V$.
Look at the set $V_0 = \{ (x,y), x\in V, y \in L_x, F_x(y) = y \}\setminus V$. It is a closed analytic subspace of $X$ which does not intersect $V$. Hence $V_0 \cdot V =0$. This gives a contradiction, as we assumed we had contracted all curves orthogonal to $\theta^*$.
\end{proof}

\begin{proof}[Proof of Proposition~\ref{p:halphen}]

Let $X$ be a generalized surface of Halphen type with $h^0(-nK_X) =1$ for all $n$  and let $\Gamma$ be its anti-canonical divisor. By~\cite{sakai}, the surface $X$ is obtained by blowing up a configuration of nine points in $\mathbb{P}^2$, and there are $22$ possible configurations.
\begin{lem}
 Let $C\subset X$ be any compact analytic curve on a smooth compact complex surface.
Consider $\pi: X' \to X$ the blow-up of $X$ at a point $p\in C$, and $C'$ the strict transform of $C$.
Then the map $\pi_1(X\setminus C) \to \pi_1(X'\setminus C')$ given by the inclusion map $\pi^{-1}: X \setminus C \to X'\setminus C'$ is surjective.
\end{lem}
\begin{lem}\label{l:lefs}
Suppose $V$ is a possibly reducible curve in $\mathbb{P}^2$.
Then for any line $H\subset \mathbb{P}^2$ which intersects $V$ transversely at smooth points, the map $\pi_1(H \setminus V) \to \pi_1(\mathbb{P}^2 \setminus V)$ is surjective.
\end{lem}
The first lemma is a consequence of the connectedness of the fibers of the map $\pi: X'\setminus C' \to X \setminus C$, whereas the second is a classical result of Lefschetz  (see~\cite{kulikov}).

\smallskip

Suppose first that $V$ is irreducible (case Ell$1$, Add$1$, Mul$1$ of~\cite[Appendix~B]{sakai}). Then $X$ is obtained by blowing up $9$ points on a irreducible cubic curve $V_0$ in $\mathbb{P}^2$.  Note that moving the position of the nine points on $V$ does not change the topological type of 
$X$ and $X\setminus V$, so that we may assume that three of these points are collinear and belong to some line $H$ intersecting $V$ transversely only at smooth points.
Thanks to Lemma~\ref{l:lefs}, $\pi_1(\mathbb{P}^2 \setminus V_0)$ is generated by small loops 
$\gamma$ included in $H$ and turning around an intersection point $p \in L \cap V_0$.
When $p$ is blown-up, the strict transform $H'$ of $H$  intersect the exceptional divisor at a point
$q$ which does not belong to $V$. We can hence find a homotopy of the pull-back of $\gamma$ to the constant path $q$ in $X\setminus V$.
Now pick a loop in $X\setminus V$. The exceptional divisor of the projection $\pi: X \to \mathbb{P}^2$ have codimension $2$, hence one can deform it to a loop which avoids $\pi^{-1} (V_0)$. 
It is homotopic in $\mathbb{P}^2 \setminus V_0$ to a sum of small loops in $H$ as above. The lift of these loops in $X\setminus V$ is homotopic to $0$. Hence $X\setminus V$ is simply connected.

In general $X$ is obtained by blowing up $9$ points (possibly infinitely near) on a curve $V_0$ in $\mathbb{P}^2$. Suppose one can find for each irreducible component $C$ of $V_0$ at least 
$\deg(C)$ distinct smooth points each of which is blown up only once (that is its preimage in $X$ is an irreducible exceptional $(-1)$ curve). Then the same argument
as before applies. This shows that $X\setminus V$ is simply connected in all cases but Mul$7$, Mul$8$, Mul$9$, Mul$10$, Add$6$, Add$7$, Add$8$, Add$9$, Add$10$, Add$11$.

Add$11$ is obtained by blowing up $9$ points on a line whose complement in $\mathbb{P}^2$ is simply connected. Hence
$X\setminus V$ is simply connected.

Add$9$ and Add$10$ (resp. Add$6$ and Mul$7$) are obtained by a sequence of blow-ups on the union of two (resp. three) lines, and there exists a point (resp. two on two distinct components) on the regular locus which is blown up only once. Again $X\setminus V$ is simply connected.

For Mul$8$, Mul$9$ and Mul$10$, then $X$ is obtained from a toric surface $X'$ by blowing three points on three distinct torus-invariant prime divisors. Small loops around these three points do not necessarily
generate the fundamental group of $X' $ minus the torus-invariant curves, but it generates a subgroup of finite index, see~\cite{danilov}.  Hence the fundamental group of  $X\setminus V$ is finite.

In case Add$8$, we may use the description of~\cite[p. 221]{sakai}. A small loop near $p_7$ turning around the exceptional divisor is mapped to  a finite number of times a loop at $p_4$ around $x=0$ in $\mathbb{P}^2$. We conclude that $\pi_1(X\setminus V)$ is finite.
The same argument applies in the last remaining case Add$7$.
\end{proof}


\subsection{Fibrations}\label{s:fibration}

Let us begin with some generalities on fibrations. 
A fibration on a surface $X$ is a regular surjective map $\varpi: X \to B$ with $B$ a Riemann surface such that $\varpi$ has connected fibers.  A dominant rational map $F: X \to X$ preserves a fibration $\varpi$ iff there exists a holomorphic map $g: B \to B$ such that 
$\varpi \circ F = g \circ \varpi $. When $X$ is rational, then $B$ is a Riemann sphere.
\begin{prop}\label{p:fibr}
Suppose $F$ preserves a fibration.
Let $e_\varpi$ be the topological degree of the restriction of $F$ to a generic fiber of the fibration; and $e_B$ be the topological degree of $g$. Then 
\begin{equation}\label{e:compute}
e = e_\varpi\times e_B \text{ and } \la = \max \{ e_\varpi, e_B\}~.
\end{equation}
\end{prop}
\begin{proof}
The computation of the topological degree is immediate. Let $P$ be the Cartier class determined in $X$ by the class of a fiber of $\varpi$. Then $F_* P = e_\varpi P$ and $F^* P = e_B P$. Hence $\la \ge \max \{ e_\varpi, e_B \}$. If $\la^2 = e$ then this implies $\la = e_B = e_\varpi$. Otherwise there exists
only two nef classes which are either $F_*$ or $F^*$ invariant, hence $P = \theta^*$ or $P = \theta_*$. In the first case one has $\la = e_B$ and in the second case $\la = e_\varpi$.
This shows~\eqref{e:compute}.
\end{proof}
An isotrivial fibration is a fibration for which any two smooth fibers are analytically diffeomorphic.
Basic examples are of the following type.
Pick a Riemann surface $C$ of genus $\ge 1$, an automorphism $h$ of $C$ of finite order, and $n$ a multiple of the order of $h$. Denote by $\Delta$ the unit disk in $\C$. Pick a primitive $n$-th root of unity $\zeta$, and define the space $X(h,n)$ to be the quotient of $\Delta \times C$ by the relation $(t,z) \simeq (\zeta t, h(z))$.
It is a complex analytic space with quotient singularities. The map $(t,z) \mapsto t^n$ descends to the quotient and induces an isotrivial fibration on
$X(h,n)$. Except for the central fiber which is isomorphic to $C/h$, the other ones are isomorphic to $C$.

Denote by $\Delta^*$ the complement of $0 \in \Delta$, and $X^*(h,n)$ the complement of the central fiber in $X(h,n)$. The following well-known result is left to the reader (for a proof consider the restriction of $\pi$ over $\Delta^*$ and lift the fibration to the universal cover of $\Delta^*$).
\begin{lem}\label{lem:str-iso}
Suppose $\pi: V \to \Delta$ is an isotrivial fibration, locally trivial over $\Delta^*$ whose generic fiber $C$
has genus at least $1$. Then there exist  an automorphism $h$ of $C$ of finite order, a multiple $n\in \N^*$ of its order, and a bimeromorphic map
$\Phi: X(h,n) \dashrightarrow V$ that is regular over $X^*(h,n)$ and such that $\pi \circ \Phi : X(h,n) \to \Delta$ is the natural fibration.
\end{lem}
The following result was already noticed in~\cite{cantat}.
\begin{prop}\label{p:isotrivial}
Pick $C$ a curve of genus $\ge 1$, $h,h '$ automorphisms of finite order, and $n,n'$ integers dividing the order of $h$ and $h'$ respectively. Any generically finite rational map  $F: X(h,n) \dashrightarrow X(h',n')$ that maps fibers to fibers extends holomorphically through the central fiber.
\end{prop}
\begin{proof}
By assumption,  $F$ induces an unramified finite cover from $X^*(h,n)$ to $X^*(h',n')$.
On the other hand, one has two natural finite covers $\Delta^* \times C$ to $X(h,n)$ and $X(h',n')$. Using the lifting criterion for coverings, we see that
for any multiple $k$ of $n'$ the composition map
$$\xymatrix{
&&&
\Delta^* \times C \ar[d]
\\
\Delta^* \times C \ar[r]_{(t^k,z)}\ar@{.>}[urrr]^{\hF}
&
\Delta^* \times C \ar[r]
&
X^*(h,n) \ar[r]^{F}
&
X^*(h',n')
}$$
lifts as a holomorphic map $\hF$ through the finite cover $\Delta^* \times C \to X(h',n')$.
By assumption $F$ is a rational map, hence the closure of the graph of $\hF$ extends analytically
over the central fiber. This implies $\hF$ is also a rational map from $\Delta^*\times C$ to itself.
Since $C$ is not a rational curve, $\hF$ is necessarily holomorphic. This implies that $\hF$ descends to a holomorphic map from $X(h,n)$ to $X(h',n')$.
\end{proof}


\subsection{Proof of Theorems~A and~D}

We begin with Theorem~A.
\begin{proof}[Proof of Theorem~A]
Take $F: \mathbb{P}^2 \dashrightarrow \mathbb{P}^2$ a dominant rational map with $e < \la^2$. Suppose first $F$ admits a holomorphic model. Then Proposition~\ref{p:cartiert*} shows that both classes $\theta^*$ and $\theta_*$ are Cartier.
For the converse, suppose $\theta^*$ and $\theta_*$ are Cartier. By Theorem~\ref{t:contract}, either $F$ admits a holomorphic model, or $F$ preserves a fibration which is induced by $\theta^*$.
Suppose we are in the second case. By assumption, $F$ is not a skew product, hence the genus of the generic fiber is $\ge 1$. Following the notation of the previous section, we have $\la = e_B> e_\varpi$.

When the genus is at least two, then the restriction map from a generic fiber to its image is an analytic diffeomorphism.  Since the topological degree of the action on the base is $\la \ge 2$, the fibration is then isotrivial. In the elliptic case too, the fibration is isotrivial.
\begin{lem}\label{lem:isotrivial}
Suppose $\varpi: X \to B \simeq \P^1$ is an elliptic fibration on a projective surface,
$F: X \dashrightarrow X$ is a rational dominant map, and
$g: B \to B$ a holomorphic map such that $\varpi \circ F = g \circ \varpi$.
If $g$ has degree $\la\ge 2$ then the fibration is isotrivial.
\end{lem}
A proof is given below.

Now suppose $C$ is a singular fiber. By Lemma~\ref{lem:str-iso}, the fibration in $X\setminus C$ is locally isomorphic near $C$ to some $X^*(h,n)$ as described in Section~\ref{s:fibration}. We then replace each singular fiber of the fibration by its local quotient model $X(h,n)$. Thanks to Proposition~\ref{p:isotrivial}, the map $F$ is holomorphic in this model.
\end{proof}

\begin{proof}[Proof of Lemma~\ref{lem:isotrivial}]
There is a natural map $J: B \to \P^1$ mapping $t \in B$ to the $j$-invariant of the fiber
$C_t =\varpi^{-1}\{ t \}$.
This map is rational. We proceed by contradiction assuming that $J$ is non constant.
Consider the set of poles of $J$: it corresponds to fibers whose local monodromy is infinite, see~\cite[p.210]{BPHV}. This implies the set of poles of $J$ to be totally invariant.
In the sequel we fix $0 \in B$ such a pole, and assume the fibration is relatively minimal.
Since $0$ is totally invariant by $g$ which has degree $\la$, it is also a critical point, hence
$B$ is isomorphic to $\P^1$. Moreover the divisor $JF$ is at least $\la -1$ times $C_0$.

By~\cite[p.213]{BPHV} the canonical bundle is proportional to the class of a fiber $K_X = a [C_t]$.
From the numerical equality $(\la -1)[C_t] \le  JF = K_X - F^*K_X = a \, (1- \la)[C_t]$, we get that $a<0$.
In particular $\mathrm{kod}(X) = - \infty$.

If the surface $X$ is ruled over a non rational base, the base is elliptic and $X$ is birational to the product $\P^1\times E$ for some fixed elliptic curve. This shows that all fibers of $\varpi$ are isogeneous to $E$, hence $\varpi$ is isotrivial. This contradicts our standing assumption. Whence $X$ is rational, and in this case one has $K_X = - \frac1m [C_t]$ where $m$ is the multiplicity of the unique multiple fiber when it exists, and $m=1$ otherwise. We conclude that
$m =1 $ and as divisors $JF = (\la-1) C_0$. This would imply $g$ to have a single ramification point which is impossible.
\end{proof}

Before proving Theorem~D, we make some general remarks on holomorphic models.
Fix $F: \mathbb{P}^2 \dashrightarrow \mathbb{P}^2$ any dominant rational map.
Recall that 
%
a holomorphic model for $F$ is a model $\pi: X\to \mathbb{P}^2$ on which the induced map $F_X \= \pi^{-1} \circ F \circ \pi$ is holomorphic.
\begin{lem}\label{l:induct}
The set of all holomorphic models for $F$ forms an inductive set. In other words,  for any two holomorphic models $X, X'$, there exists a third one $X''$ which dominates both $X$ and $X'$.
\end{lem}
\begin{proof}
Define $X''$ to be the graph of the map $\varpi \=  \pi^{-1} \circ \pi': X ' \to X$ in $X' \times X$.
Then we have two natural projection maps $\pi_1: X'' \to X'$ and $\pi_2: X '' \to X$ that are bimeromorphic, hence $X''$ dominates both $X$ and $X'$.
On $X''$, the map induced by $F$ is the restriction of the product map $(F_{X'}, F_X)$ on $X' \times X$.
It is thus holomorphic.
\end{proof}
~

 \begin{proof}[Proof of Theorem~D]
Suppose first that  $\theta^*$ is not associated to a fibration.
 Take any holomorphic model $X$ such that the dimension of $\NS(X)$ is minimal among all holomorphic models. 
Pick any other holomorphic model $X'$. By Lemma~\ref{l:induct}, there exists a third holomorphic model $X''$ such that the natural maps $\varpi: X'' \to X$ and $\varpi' : X'' \to X'$ are both regular.  Apply Theorem~\ref{t:contract} to $X''$.  By our working assumption on $\theta^*$ we may contract all curves orthogonal to $\theta^*$ in $\NS(X'')$. This yields a holomorphic model $X_0$ for $F$.
By Proposition~\ref{p:cartiert*}, the class $\theta^*$ is Cartier and determined in $X$, hence $\theta^*$ is orthogonal to all $\varpi$-exceptional curves. This means the map $X \to X_0$ is regular, hence by~\eqref{e:small-decomp} $\dim \NS(X_0 )\le \dim \NS(X)$. But $X$ was chosen with $\dim \NS(X)$ minimal, so that $X = X_0$.
On the other hand, $\theta^*$ is also orthogonal to all $\varpi'$-exceptional curves 
hence $X' \to X_0 = X$ is regular. This concludes the proof.

 \medskip
 
 Suppose now that $\theta^*$ is  associated to a fibration. 
As before pick a holomorphic model $X$ of $F$. Then there exists a fibration $\varpi : X \to \mathbb{P}^1$ which is determined by $\theta^*$ and preserved by $F$. The induced map $g$ on the base has topological degree $e_B  =\la \ge 2$ by assumption.

The set of reducible fibers is finite and totally invariant by $g$. Replacing $F$ by an iterate, we may assume that any prime divisor $C$ of a reducible fiber is also totally invariant so that $F^* C = \la C$. Whence $e\times  C^2 = (F^* C)^2 =\la^2 \times C^2$, which forces $C^2 =0$
as $e < \la^2$. In other words, $\varpi$ admits no reducible fiber. 

Now suppose $F$ admits a holomorphic model $X'$ dominating $X$. The same argument proves that the lift of the fibration to $X'$ has no reducible fiber either which implies $X'$ to be isomorphic to $X$. 
By Lemma~\ref{l:induct}, we conclude that $F$ admits a unique holomorphic model.
This is a stronger statement than the conclusion of Theorem~D.
\end{proof}


\subsection{Skew products}\label{s:skew}
Recall that a skew product is a rational map of $\mathbb{P}^1 \times \mathbb{P}^1$ which preserves
the rational fibration induced by the projection onto the first factor. Theorem~A fails for skew product as we shall explain now.

Let us introduce some notation. Let $\varpi:  \mathbb{P}^1 \times \mathbb{P}^1 \to \mathbb{P}^1$ be the  projection onto the base factor. For $x \in \mathbb{P}^1$, we denote by $L_x$ the fiber $\varpi^{-1}\{ x \}$.
If $F$ is a skew product, we let $g: \mathbb{P}^1 \to \mathbb{P}^1$ be the induced map on the base, and $F_x: L_x \to L_{g(x)}$ the natural map on the fiber.
As in Section~\ref{s:fibration}, write  $e_B$ for the topological degree of $g$ and $e_\varpi$ for the topological degree of $F_x$ for a generic $x$. We let $\mathrm{R}$ be the finite set of points $x \in \mathbb{P}^1$ such that  $\deg(F_x) < e_\varpi$.

 Suppose $e_B > e_\varpi$. Then $e < \la^2$ and the class $\theta^*$ is determined by a fiber $L_x$.
Suppose moreover that there exists no contracted curves, i.e. for any $x \in \mathrm{R}$ the map $F_x$ has degree $\ge 1$. Then by~\cite[Corollary~2.6]{deggrowth}, the space of Cartier classes determined in $\mathbb{P}^1 \times \mathbb{P}^1$ is preserved by $F_*$. We conclude that $\theta_*$ is Cartier and determined in  $\mathbb{P}^1 \times \mathbb{P}^1$. Finally assume $\mathrm{R}$ is not totally invariant by $g$. Then $R$ has an infinite backward orbit so that infinitely many points are eventually mapped by $F$ to a point of indeterminacy lying over $R$. Whence $F$ admits no holomorphic model.

A concrete example is given by $F(x,y)=(x^3+1, \frac{xy^2 -y}{y^2 -x})$ in $\C^2$. Then $e_B = 3 > e_\varpi = 2$, $g(x) = x^3 +1$ and $\mathrm{R} = \{ 0 , 1 , j, j^2\}$ with $j$ a primitive $3$-root of unity. This shows Theorem~A does not hold for skew products.


\end{document}